\documentclass[a4paper, 10pt]{article}

\usepackage{amsthm}
\usepackage{amsmath}
\usepackage{amssymb}

\usepackage{todonotes}
\usepackage[shortlabels]{enumitem}
\usepackage{url}
\usepackage[bbgreekl]{mathbbol}
\usepackage{xcolor}

\newcommand{\leszek}[1]{#1}

\newtheorem{theorem}{Theorem}[section]
\newtheorem{lemma}[theorem]{Lemma}
\newtheorem{corollary}[theorem]{Corollary}
\newtheorem{proposition}[theorem]{Proposition}

\newtheorem*{mainthm}{Theorem \ref{thm:wkl-iso} (abbreviated version)}

\theoremstyle{definition}
\newtheorem{definition}[theorem]{Definition}

\theoremstyle{remark}
\newtheorem*{remark}{Remark}

\newtheorem{question}[theorem]{Question}

\let\le\leqslant
\let\ge\geqslant
\newcommand{\defeq}{\mathrel{\mathop:}=}

\newcommand{\N}{\mathbb{N}}

\newcommand{\X}{\mathcal{X}}
\newcommand{\Y}{\mathcal{Y}}

\newcommand{\W}{\mathcal{W}}

\let\extset\widetilde

\newcommand{\eb}[2]{\exists #1 \! \le \! #2 \,}
\newcommand{\ab}[2]{\forall #1 \! \le \! #2 \,}

\newcommand{\ee}{\mathrm{e}}

\newcommand{\RCA}{\mathrm{RCA}}
\newcommand{\ACA}{\mathrm{ACA}}
\newcommand{\WKL}{\mathrm{WKL}}

\newcommand{\RT}{\mathrm{RT}}
\newcommand{\COH}{\mathrm{COH}}
\newcommand{\CAC}{\mathrm{CAC}}
\newcommand{\ADS}{\mathrm{ADS}}

\newcommand{\ind}{\mathrm{I}}
\newcommand{\bd}{\mathrm{B}}
\newcommand{\PA}{\mathrm{PA}}

\newcommand{\tuple}[1]{\langle#1\rangle}

\newcommand{\gn}[1]{\ulcorner#1\urcorner}

\newcommand{\Sat}{\mathrm{Sat}}
\newcommand{\Ack}{\mathrm{Ack}}

\newcommand{\rs}{\mathrm{r}\Sigma^1_1}
\newcommand{\rp}{\mathrm{r}\Pi^1_2}

\newcommand{\ignore}[1]{}

\definecolor{colormfc}{HTML}{BFE3EC}

\title{An isomorphism theorem \\ for models of Weak K\"onig's Lemma \\ without primitive recursion}
\author{
Marta Fiori-Carones\thanks{Institute of Mathematics, University of Warsaw,
\texttt{marta.fioricarones@outlook.it}, \texttt{lak@mimuw.edu.pl}.}
\and Leszek Aleksander Ko{\l}odziejczyk\footnotemark[1]
\and Tin Lok Wong\thanks{Department of Mathematics, National University of Singapore, \texttt{matwong@nus.edu.sg}. }
\and Keita Yokoyama\thanks{Mathematical Institute, Tohoku University, \texttt{keita.yokoyama.c2@tohoku.ac.jp}.}
}

\begin{document}

\maketitle

\begin{abstract}
We prove that if $(M,\X)$ and $(M,\Y)$ are countable models of the theory $\WKL^*_0$
such that $\ind \Sigma_1(A)$ fails for some $A \in \X \cap \Y$, then
$(M,\X)$ and $(M,\Y)$ are isomorphic. As a consequence, the analytic hierarchy
collapses to $\Delta^1_1$ provably in $\WKL^*_0 + \neg \ind \Sigma^0_1$,
and $\WKL$ is the strongest $\Pi^1_2$ statement that
is $\Pi^1_1$-conservative over $\RCA^*_0 + \neg \ind \Sigma^0_1$.

Applying our results to the $\Delta^0_n$-definable sets
in models of $\RCA^*_0 + \bd \Sigma^0_n + \neg \ind \Sigma^0_n$
that also satisfy an appropriate relativization of Weak K\"onig's Lemma,
we prove that for each $n \ge 1$, the set of $\Pi^1_2$ sentences
that are $\Pi^1_1$-conservative over $\RCA^*_0 + \bd \Sigma^0_n + \neg \ind \Sigma^0_n$
is c.e.
In contrast, we prove that the set of $\Pi^1_2$ sentences
that are $\Pi^1_1$-conservative over $\RCA^*_0 + \bd \Sigma^0_n$
is $\Pi_2$-complete. This answers a question of Towsner.

We also show that $\RCA_0 + \RT^2_2$ is $\Pi^1_1$-conservative over $\bd \Sigma^0_2$
if and only if it is conservative over $\bd \Sigma^0_2$ with respect to $\forall \Pi^0_5$ sentences.
\end{abstract}

\leszek{
In this paper, we investigate a new model-theoretic argument that can be used to study some fragments of second-order arithmetic. It is a clich\'e that classical model-theoretic techniques are of limited use in understanding models of arithmetic, and in particular it is well-known that commonly studied theories of arithmetic do not have nice model-theoretic properties such as model completeness and quantifier elimination. 
Nevertheless, we use an automorphism-based argument to obtain a kind of model completeness result for a specific fragment of second-order arithmetic with (partially) negated induction.
This result can be applied to obtain new information both about general properties of some theories of second-order arithmetic, including true ones, and about specific principles considered in reverse mathematics.
}

The main theorem of \leszek{the paper} concerns \leszek{the} axiomatic theory known as $\WKL^*_0$.
This is a fragment of second-order arithmetic consisting of the theory $\RCA^*_0$ --
that is, $\Delta^0_1$-comprehension, $\Delta^0_1$-induction, and the totality of exponentiation --
and the additional axiom $\WKL$ \leszek{expressing a compactness principle called} Weak K\"onig's Lemma,
which says that any infinite 0--1 tree has an infinite path. Compared to the usual
base theory considered in reverse mathematics, $\RCA_0$, the system $\WKL^*_0$
is proof-theoretically much weaker, as a result of not requiring induction for $\Sigma^0_1$ properties.
On the other hand, $\WKL^*_0$ goes beyond $\RCA_0$
in that $\WKL$ implies the existence of noncomputable sets.

It is known that in some important respects $\WKL^*_0$ behaves similarly
to its stronger cousin $\WKL_0 := \RCA_0 + \WKL$.
For example, it was shown in the seminal paper \cite{simpson-smith}
that $\WKL^*_0$ is $\Pi^1_1$-conservative over $\RCA^*_0$.
This is analogous, and can be proved analogously, to the well-known theorem
that $\WKL$ is $\Pi^1_1$-conservative over $\RCA_0$.
A more recent observation \cite{enayat-wong} is that $\WKL^*_0$, like $\WKL_0$,
implies the completeness theorem for first-order logic; moreover,
it also implies some watered down versions of completeness tailored to cut-free consistency,

which are often helpful as $\WKL^*_0$ does not imply cut elimination for first-order logic.
This makes it possible to discover numerous connections between the model theory of $\WKL^*_0$
and that of first-order arithmetic.

Here, we prove a result that applies specifically to models of $\WKL^*_0$ without $\Sigma^0_1$-induction
and has no apparent analogue for $\WKL_0$. It can be seen as an extension to the second-order setting of some earlier results to the effect that countable models satisfying $\Sigma_n$-collection (and $\exp$) but not $\Sigma_n$-induction
have (in fact, many) nontrivial automorphisms.

\begin{mainthm}
Let $(M, \X)$ and $(M,\Y)$ be countable models of $\WKL^*_0$
such that $(M,\X \cap \Y) \models \neg \ind \Sigma^0_1$.
Then $(M, \X)$ and $(M,\Y)$ are isomorphic.
\end{mainthm}

One possible interpretation of the theorem is that the only new sets
that can be added to a model of $\RCA^*_0 + \neg \ind \Sigma^0_1$
are paths through binary trees. To see this, note that if $(M, \X)$ is a countable
model of $\RCA^*_0 + \neg \ind \Sigma^0_1$, then by
\cite{simpson-smith} it can be $\bbomega$-extended
(that is, extended without changing the first-order universe $M$)
to a model $(M, \Y)$ of $\WKL^*_0$. If $G \subseteq M$ is an arbitrary
set contained in any other $\bbomega$-extension of $(M, \X)$ satisfying $\RCA^*_0$,
then even though $\Y$ might not contain $G$ itself,
by Theorem 2.1 it will contain a set $H$ such that $(M, G)$ and $(M, H)$ are isomorphic.

\leszek{
In its more general form, Theorem \ref{thm:wkl-iso} allows the isomorphism
to fix a given finite tuple of first- and second-order elements.
This has a number of consequences
both for $\WKL^*_0$ and for other theories.
}

\leszek{
For example, it follows from Theorem \ref{thm:wkl-iso} that provably in $\WKL^*_0 + \neg \ind \Sigma^0_1$ the analytic hierarchy collapses down to $\Delta^1_1$, and even to a slightly more restricted class.
If one views arithmetical formulas (i.e., those without second-order quantifiers) as quantifier-free,
this means that $\WKL^*_0 + \neg \ind \Sigma^0_1$ is a model complete theory, 
and it is the model companion of $\RCA^*_0 + \neg \ind \Sigma^0_1$.
As a result, $\WKL$ is the strongest $\Pi^1_2$ statement that is $\Pi^1_1$-conservative
over $\RCA^*_0 + \neg \ind \Sigma^0_1$.
The model completeness phenomenon is quite unexpected, as it 
does not typically occur in any recognizable form among fragments of first- or second-order arithmetic.
}



When $n \ge 1$ is arbitrary, and $(M,\X)$ is a countable model of $\RCA^*_0$ satisfying the $\Sigma^0_n$-collection scheme $\bd\Sigma^0_n$ and a relativization of $\WKL$ to $\Delta^0_n$-definable sets but not satisfying $\ind\Sigma^0_n$,
then the family of $\Delta^0_n$-definable sets of $(M,\X)$ forms a model of $\WKL^*_0 + \neg \ind \Sigma^0_1$.
Thus, Theorem \ref{thm:wkl-iso} applies to this family, implying that the variety of $\bbomega$-extensions
of models of $\RCA^*_0 + \bd \Sigma^0_n + \neg \ind \Sigma^0_n$ is also somewhat limited,
though in general not quite as drastically as for $n=1$.

A corollary of this is related to a question of Towsner \cite{towsner:maximum-conservative}.
Towsner proved that the set of $\Pi^1_2$ sentences that are $\Pi^1_1$-conservative over $\RCA^*_0 + \ind \Sigma^0_n$
is $\Pi_2$-complete for each $n \ge 1$, and he asked whether this still holds true if $\ind \Sigma^0_n$ is replaced by
$\bd \Sigma^0_n$. We show that this is, rather surprisingly, not the case if additionally one explicitly
negates $\ind \Sigma^0_n$: for $n \ge 1$, the set of $\Pi^1_2$ sentences that are $\Pi^1_1$-conservative
over $\RCA^*_0 + \bd \Sigma^0_n + \neg \ind \Sigma^0_n$ is computably enumerable.
On the other hand, we show that Towsner's question as originally stated has a positive answer.
The argument for this does not rely on Theorem \ref{thm:wkl-iso}, although one of its ingredients is a result
(Theorem \ref{thm:neg-is1-cs2}) whose known proofs are also based on automorphisms of models of arithmetic.

One further consequence of Theorem \ref{thm:wkl-iso} is connected to Ramsey's theorem for pairs and two colours,
$\RT^2_2$, and more precisely to the problem whether $\bd\Sigma^0_2$ axiomatizes the $\Pi^1_1$ consequences of $\RCA_0 + \RT^2_2$. We prove that this is the case if and only if $\bd \Sigma^0_2$ proves all the $\forall \Pi^0_5$ consequences
of $\RCA_0 + \RT^2_2$ (note that this is known for $\forall \Pi^0_3$ consequences
and open from $\forall \Pi^0_4$ onwards); moreover, if this is the case then it can be proved using the ``single jump control'' method of \cite{cholak-jockusch-slaman}. We prove these facts by applying
Theorem~\ref{thm:wkl-iso} to the $\Delta^0_2$-definable sets in models of $\RCA_0 + \bd \Sigma^0_2 + \neg \ind \Sigma^0_2$.
A different proof is possible using the fact that $\RT^2_2$ is a so-called restricted $\Pi^1_2$ formula.

The remainder of this paper is structured as follows. We present the necessary definitions and background
in the preliminary Section \ref{sec:prelim}. Section \ref{sec:iso} contains the proof of our isomorphism theorem.
In Sections \ref{sec:wklstar} and \ref{sec:generalization}, we discuss the consequences of the theorem for $\WKL^*_0$
and for higher levels of the arithmetic hierarchy, respectively.
Section \ref{sec:rt22} concerns implications of the theorem for $\RT^2_2$,
as well as a more general discussion of the behaviour of restricted $\Sigma^1_1$ formulas
under negated $\Sigma^0_1$-induction.
Finally, in Section \ref{sec:towsner} we present our positive solution to Towsner's problem.

\section{Preliminaries} \label{sec:prelim}








We assume that the reader has some familiarity with fragments of second-order arithmetic
and with models of first- and second-order arithmetic (see \cite{simpson:sosoa} or \cite{hirschfeldt:slicing} for second-order arithmetic and \cite{Kaye91} for models of first-order arithmetic).

We write $\Delta^0_n$, $\Sigma^0_n$, $\Pi^0_n$ to denote the usual formula classes defined
in terms of first-order quantifier alternations, but allowing second-order free variables.
On the other hand, notation without the superscript $0$, like $\Delta_n$, $\Sigma_n$, $\Pi_n$,
represents analogously defined classes of purely first-order, or ``lightface'', formulas,
that do not contain any second-order variables at all. If we want to specify the second-order parameters appearing
in a $\Sigma^0_n$ formula, we use notation like $\Sigma_n(A)$.
We extend these conventions to naming theories.

If $\Gamma$ is a class of formulas, then the class $\forall \exists \Gamma$
contains formulas that consist of a block of universal (first- and/or second-order) quantifiers,
followed by a block of existential quantifiers, followed by a formula from $\Gamma$.
The class $\forall \Gamma$ is defined analogously.
For example, $\forall\Sigma^0_n$ and $\forall\Pi^0_{n+1}$ are the same class of $\Pi^1_1$ formulas.

The theory $\RCA^*_0$, originally defined in
\cite{simpson-smith}, is obtained from $\RCA_0$ by weakening the $\Sigma^0_1$-induction axiom $\ind \Sigma^0_1$
to $\Delta^0_1$-induction and adding a $\Pi_2$ axiom $\exp$ that explicitly guarantees the totality of exponentiation. The theory $\WKL^*_0$ is obtained from $\WKL_0$ in an analogous way;
put differently, $\WKL^*_0$ is $\RCA^*_0$ plus Weak K\"onig's Lemma $\WKL$.
Already $\RCA^*_0$ proves the collection scheme $\bd\Sigma^0_1$, and the first-order consequences of $\RCA^*_0$ and of
$\WKL^*_0$ are axiomatized by $\bd\Sigma_1 + \exp$.

Already $\ind\Delta^{0}_{0} + \exp$ is strong enough
to support a well-behaved universal $\Sigma_n(Y)$ formula $\mathrm{Sat}_n(x,y,Y)$
(``the $\Sigma^0_n$ formula (with G\"odel number) $x$ with one first- and one second-order variable
holds of the number $y$ and the set $Y$''), for each standard natural number $n$.
As a consequence, we have access to these universal formulas
in any theory containing $\RCA^*_0$.

When we consider a model $(M,\X)$ of some fragment of second-order arithmetic
(or simply work inside this fragment without reference to a specific model),
the word \emph{set} without any qualifier refers to an element of the second-order universe $\X$.
In contrast, a \emph{$\Sigma^0_n$-definable set} (or \emph{$\Sigma^0_n$-set} for brevity)
is any subset of the first-order universe $M$ that is definable in $(M,\X)$ by $\Sigma^0_n$ formula.
A \emph{$\Delta^0_n$-definable set} or \emph{$\Delta^0_n$-set} is a $\Sigma^0_n$-set that is simultaneuosly definable by a $\Pi^0_n$ formula. The notions of a $\Sigma_n$- and $\Delta_n$-set are defined in the obvious way.

In general, the models we study only satisfy $\Delta^0_1$-comprehension, so
$\Delta^0_n$-sets for $n \ge 2$ and $\Sigma^0_n$-sets for $n \ge 1$ will not always be sets.
However, thanks to the availability of universal formulas,
we can quantify over $\Delta^0_n$- or over $\Sigma^0_n$-sets using second-order quantifiers
(e.g.~``for every $Y$ and every equivalent pair of a $\Sigma_n(Y)$ and a $\Pi_n(Y)$ formula...'').
Whenever we present reasoning or statements that are
to be understood as formalized in second-order arithmetic,
we indicate quantification over definable sets that might not be sets by putting a tilde
over the quantified variable, as in $\extset X$. So, for example,
``for every $\Sigma^0_1$-set $\extset X$ there exists $X$ such that
$\forall k\, (k \in \extset X \leftrightarrow k \in X)$'' would be a slightly unusual
way of expressing the $\Sigma^0_1$-comprehension axiom of $\ACA_0$.
Sometimes, in the context of a model-theoretic argument, we may
also use a tilde to warn the reader that a given subset of $M$ might not
belong to the second-order universe at hand.
However, in those situations we treat adding the tilde as a matter of convenience,
and we do not strive for consistency (partly because it would be unattainable:
often we work with several models at the same time, and what is a set in one model
may not be a set in another).

We write $\Delta^0_n\text{-}\mathrm{Def}(M,\X)$ for the collection of all the $\Delta^0_n$-sets of $(M,\X)$.
If $A$ is a subset of $M$, we write $\Delta^0_n\text{-}\mathrm{Def}(M,A)$ for the collection of all the $\Delta_n(A)$-sets.
If $(M,\X)\models\bd\Sigma^0_n+\exp$ where $n\ge1$,
then $(M,\Delta^0_n\text{-}\mathrm{Def}(M,\X))\models\RCA^*_0$. A model of $\RCA^*_0$ is \emph{topped}
if it has the form $(M, \Delta^0_1\text{-}\mathrm{Def}(M,A))$ for some $A\subseteq M$.

The notation $A \le_\mathrm{T} B$ means that $A$ is $\Delta_1(B)$-definable.
The \emph{join} of $A$ and $B$, denoted by $A \oplus B$, is $\{\langle{0,k}\rangle:k \in A\} \cup \{\langle{1,k}\rangle:k \in B\}$;
note that $A \le_\mathrm{T} A \oplus B$ and $B \le_\mathrm{T} A \oplus B$,
but at the same time $A \oplus B$ is $\Delta_1(A,B)$-definable.
$A \equiv_\mathrm{T} B$ means that $A \le_\mathrm{T} B$ and $B \le_\mathrm{T} A$.
If $A$ is a subset of some model $M$ and $(M,A) \models \ind \Delta_0(A) + \exp$,
then $A^{(n)}$, the \emph{$n$-th jump} of $A$, is the $\Sigma_n(A)$-set
$\{\langle e, k \rangle: (M,A)\models\Sat_n(e,k,A)\}$. As usual, we write $A'$ for $A^{(1)}$, the \emph{jump} of $A$.
The following result provides a variant of the important method of ``jump inversion''
adapted to the setting of fragments of second-order arithmetic.

\begin{theorem}[Belanger~\cite{belanger:coh}]\label{thm:belanger-jump-inversion}
Let $n \ge 1$. Assume that $M$ is a structure and $A, C \subseteq M$ are such that $(M, A) \models \bd \Sigma^0_{n+1}$
and $(M, A' \oplus C) \models \bd \Sigma^0_n$. Then there exists $B \subseteq M$ such that $C$ is $\Delta_2(A \oplus B)$-definable and $(M, A \oplus B) \models \bd \Sigma^0_{n+1}$.
\end{theorem}

A \emph{cut} in a model of arithmetic $M$ is any subset $I \subseteq M$
that contains $0$ and is closed downwards and under successor.
We can write $I \subseteq_\ee M$ to indicate that $I$ is a cut in $M$.
Note that if $(M,\X) \models \RCA^*_0$,
then no proper cut in $M$ can be a member of $\X$.
On the other hand, some proper cut is $\Sigma^0_1$-definable in $(M,\X)$
if and only if $(M,\X) \models \neg \ind \Sigma^0_1$.
If $I$ is a $\Sigma^0_1$-cut in $(M,\X)$,
then there is an \emph{unbounded} (i.e.~cofinal in $M$) set $A \in \X$
that can be enumerated in $\X$ in increasing order as $A = \{a_i: i \in I\}$.

For an element $s$ of a model $M$, $\Ack(s)$ stands for
$\{a \in M: M \models a \in_\Ack s\}$, where $\in_\Ack$
is the usual Ackermann interpretation of set theory in arithmetic
(``the $a$-th bit in the binary notation for $s$ is 1'').
In a model $(M,\X)$ of $\RCA^*_0$, sets of the form $\Ack(s)$
are exactly the bounded subsets of $M$ that belong to $\X$.
The following theorem states an important basic fact about such sets.

\begin{theorem}[{Chong--Mourad~\cite[Proposition~4]{chong-mourad:degree-cut}}]\label{lem:chong-mourad}
Let $(M,\X) \models \RCA^*_0$.
Then for every pair of bounded disjoint $\Sigma^0_1$-definable sets $X,Y \subseteq M$
there exists $s \in M$ such that $\Ack(s) \cap (X \cup Y) = X$.
\end{theorem}

We use the symbol $\omega$ to denote the set of standard natural numbers
and the symbol $\N$ to denote the set of natural numbers
as formalized within $\RCA^*_0$. In other words, if $(M,\X)$ is a model of $\RCA^*_0$,
then $\N^{(M,\X)}$ is simply the first-order universe $M$.
This convention clashes with the custom of referring to a model
that has the same first-order universe as some other model
but a smaller second-order universe as an $\omega$-submodel.
Because of this, we use the term \emph{$\bbomega$-submodel}
(and, conversely, \emph{$\bbomega$-extension}) in such situations.

A structure $\mathfrak{A}$ is \emph{recursively saturated} if for any computable set of formulas
$\{\psi_n(\overline x, \overline y): n \in \omega\}$ and any tuple $\overline b$ of the appropriate length,
if $\mathfrak{A} \models \exists \overline x \bigwedge_{i = 0}^n \, \psi_i(\overline x, \overline b)$ for
each $n \in \omega$, then there is a tuple $\bar a$ such that $\mathfrak{A} \models \psi_n(\overline a, \overline b)$
for all $n$. Recursively saturated structures only play a very minor role in this paper;
for more on them, see e.g.~\cite[Chapters 11.2 \& 15]{Kaye91}.

\section{The isomorphism theorem}\label{sec:iso}

In this section, we state and prove our main theorem on isomorphisms between models of $\WKL^*_0 + \neg \ind \Sigma^0_1$. One can view the theorem as a generalization to second-order arithmetic
of a result of Kossak's \cite[Theorem~3.1]{kossak:extensions} (see~\cite{kossak:extensions-correction} or \cite{kaye:properties} for a correction to the proof) saying that every countable model of $\bd \Sigma_1(A) + {\exp} + \neg \ind \Sigma_1(A)$ has continuum many automorphisms. In fact, our proof is somewhat reminiscent of Kossak's argument,
in which one also finds a truth-coding trick that goes back to
   Smory\'nski~\cite[Lemma~1.2]{smorynski},
   Kotlarski~\cite[Lemma~4.4]{kotlarski:elem-cuts-rec-sat}, and
   Alena Vencovsk\'a~[unpublished].

\begin{theorem}\label{thm:wkl-iso}
Let $(M, \X)$ and $(M,\Y)$ be countable models of $\WKL^*_0$
such that $(M,\X \cap \Y) \models \neg \ind \Sigma^0_1$. 
Let $\overline{c}$ be a tuple of elements of $M$ and $\overline{C}$ be a tuple of elements of $\X \cap \Y$.  Then there exists an isomorphism $h$ between $(M, \X)$ and $(M,\Y)$ such that $h(\overline{c})=\overline{c}$ and $h(\overline{C})=\overline{C}$.
\end{theorem}


\begin{proof}
Since $(M,\X \cap \Y) \models \neg \ind \Sigma^0_1$, there exists a set $A \in \X \cap \Y$ cofinal in $M$ and a proper cut $I \subsetneq_\ee M$  such that $A$ can be enumerated in increasing order as $A = \{a_i: i \in I\}$. We may assume w.l.o.g.~that $I$ is closed under $\exp$ (see \cite[Theorem~2.4]{hajek:interpretability} or~\cite[Lemma 9]{ky:categorical}) and that the tuple of sets $\overline C$ contains both $A$ itself and the $\Delta_1(A)$-definable set $\{\tuple{i,a_i}: i \in I\}$, so that the relation ``$x = a_i$'' between $x$ and $i$ is $\Delta_0(\overline C)$-definable. Finally, by adding a sufficiently large number to all the elements of $A$ if necessary, we may assume that $a_0 > I$.


We build the isomorphism $h$ by a back-and-forth construction.
At each step of the construction, we have finite tuples $\overline r, \overline s$, both in $M$,
and $\overline R, \overline S$, in $\X, \Y$ respectively, such that we have committed to
$h(\overline r) = \overline s$ and $h(\overline R) = \overline S$.
Initially, $\overline r = \overline s  = \overline{c}$ and $\overline R = \overline S = \overline{C}$.
In each step, we add either a first- or a second-order element to either the domain or the range of $h$,
so that after $\omega$ steps of the construction $h$ becomes a bijection from $(M,\X)$ onto $(M, \Y)$.
The inductive condition we maintain is:
\begin{enumerate}
\renewcommand{\theenumi}{\#}
\renewcommand{\labelenumi}{(\theenumi)}
\item\label{cond:hash}
there exist $b >I$ and $\epsilon > \omega$ in $M$ such that for each $i \in I$,   \\
each $j < b$, and each $\Delta_0$ formula $\delta$ (with G\"odel number) $< \epsilon$,\\
$(M, \X) \models \delta(a_i,j,\overline r, \overline R)$ iff $(M, \Y) \models \delta(a_i,j,\overline s, \overline S)$.
\end{enumerate}
Of course, the statement that the (potentially nonstandard) formula $\delta$ is satisfied in a model is expressed using a fixed truth definition for $\Delta_0$ formulas, $\mathrm{Sat}_{\Delta_0}$. Note that \eqref{cond:hash} holds at the beginning of the construction with $\tuple{\overline r, \overline R} = \tuple{\overline s, \overline S} = \tuple{\overline{c}, \overline{C}}$.
We have to verify that it can be preserved at each step of the construction.
There are two kinds of step to consider: adding a first-order element to our tuple
and adding a second-order element.

\paragraph{First-order step.}
Let us consider the case where we want to add a new first-order element $r^*$ to the domain
of our map $h$, so we need to find $s^* \in M$ such that the tuples $\overline r, r^*, \overline R$ and $\overline s, s^*, \overline S$
still satisfy \eqref{cond:hash}.  The construction for adding a first-order element to the range of $h$ is analogous.

By the inductive assumption, we have $b >I$ and $\epsilon > \omega$ witnessing \eqref{cond:hash} for $\overline r, \overline R, \overline s, \overline S$. Define $b' = \log\log\log b$ and $\epsilon' = \log \log \epsilon$. Notice that $b' >I$ and $\epsilon' > \omega$, since both $I$ and $\omega$ are closed under $\exp$.

Let $D$ stand for the following definable subset of $(I \times b' \times \epsilon')$:
\begin{multline*}
\{\tuple{i,j,\gn{\delta}} : i \in I, j < b', \gn{\delta} < \epsilon' \textrm{ is (the G\"odel number of) a } \Delta_0 \textrm{ formula,} \\ \textrm{and } (M,\X) \models \delta(a_i,j,\overline r, r^*,\overline R)) \}
\end{multline*}
Both $D$ and $(I \times b' \times \epsilon') \setminus D$ are bounded and $\Sigma^0_1$-definable, so by Theorem \ref{lem:chong-mourad} there exists $j^* \in M$ coding a set $\Ack(j^*)$ such that $\Ack(j^*) \cap (I \times b' \times \epsilon')= D$.  Moreover, since $I < b' \ll \log \log b$ and since we can choose $\epsilon$ small enough, we can assume that $j^* < \log b$.

Let $i^* \in I$ be such that $r^* \le a_{i^*}$.
By the definition of $j^*$, for each $i \in I$ the structure $(M,\X)$ satisfies:
\begin{equation}\label{eqn:char-of-r^*}
\eb{y}{a_{i^*}}\ab{i'}{i}\ab{j}{b'}\bigwedge_{\gn{\delta} < \epsilon'}\bigl(\delta(a_{i'},j,\overline r, y, \overline R)  \leftrightarrow \tuple{i',j,\gn{\delta}} \in_\Ack j^* \bigr),
\end{equation}
as witnessed by $y\defeq r^*$. If $i \in I$
and $i \ge i^*$,
then (\ref{eqn:char-of-r^*}) implies:
\begin{multline}\label{eqn:char-of-r^*1}
\eb{x, \! y}{a_i} \Biggl(x = a_{i^*} \land y \le x \land \ab{i'}{i}\ab{v}{a_i}\Bigl(v = a_{i'} \rightarrow \\
\rightarrow \ab{j}{b'}
\bigwedge_{\gn{\delta} < \epsilon'}\bigl(\delta(v,j,\overline r, y, \overline R)  \leftrightarrow \tuple{i',j,\gn{\delta}} \in_\Ack j^* \bigr)\Bigr)\Biggr).
\end{multline}

For each fixed $i \in I$, the statement (\ref{eqn:char-of-r^*1})
can be expressed as a (nonstandard) $\Delta_0$ formula $\gamma$ with parameters
$i^*, b', j^*, a_i, \overline r, \overline R$. The formula $\gamma$ consists of a fixed-length part independent of $\epsilon'$, followed by a conjunction whose conjuncts correspond to $\Delta_0$ formulas $\delta$ with $\gn{\delta} < \epsilon'$. Each such conjunct is an equivalence, with the formula $\delta$, which consists of at most $\log \epsilon'$ symbols, written out on the left-hand side and with $\gn{\delta}$ referred to by its (say binary) numeral on the right-hand side. Thus, in total, $\gamma$ has length $O(\epsilon' \cdot \log (\epsilon'))$, which means that $\gn{\gamma} < \epsilon$. Moreover, the tuple of parameters $\tuple{{i^*}, b', j^*}$ is smaller than $b$.

We can therefore apply \eqref{cond:hash} from the inductive hypothesis with $\delta \defeq \gamma$ and $j\defeq\tuple{{i^*}, b', j^*}$ in order to conclude that for each sufficiently large $i \in I$, the structure $(M, \Y)$ satisfies (\ref{eqn:char-of-r^*1}) with $\overline r, \overline R$ replaced by $\overline s, \overline S$. From this it easily follows that $(M, \Y)$ in fact satisfies
\begin{equation}\label{eqn:char-of-s^*}
\eb{y}{a_{i^*}}\ab{i'}{i}\ab{j}{b'}\bigwedge_{\gn{\delta} < \epsilon'}\bigl(\delta(a_{i'},j,\overline s, y, \overline S)  \leftrightarrow \tuple{i',j,\gn{\delta}} \in \Ack(j^*) \bigr)
\end{equation}
for each $i \in I$. \leszek{Note that (\ref{eqn:char-of-s^*}) is (\ref{eqn:char-of-r^*}) with $\overline r, \overline R$ replaced by $\overline s, \overline S$.}

By $\bd\Sigma^0_1$ in $(M,\Y)$, there must exist some $y  \le a_{i^*}$ that witnesses (\ref{eqn:char-of-s^*}) for all $i \in I$. We can choose any such $y$ as our $s^*$.

\paragraph{Second-order step.}
This step is somewhat similar to the first-order one, with the role of $\bd \Sigma^0_1$ now played by $\WKL$.
As before, we consider only the case where we want to add a new second-order element $R^* \in \X$
to the domain of $h$ and we need to find $S^* \in \Y$ such that $\overline r, \overline R, R^*$ and $\overline s, \overline S, S^*$ still satisfy \eqref{cond:hash}.

By the inductive assumption, we have $b >I$ and $\epsilon > \omega$ witnessing \eqref{cond:hash} for $\overline r, \overline R, \overline s, \overline S$. The parameters $b'$ and $\epsilon'$ are defined as in the first-order step. The set $D$ and the element $j^*$ are also defined as before, but with $\overline r,r^*, \overline R$ replaced by $\overline r, \overline R, R^*$.



It follows from the definition of $j^*$ that $(M,\X)$ satisfies
\begin{multline}\label{eqn:char-of-R^*}
\exists F  \! \subseteq \! [0, \log a_{i})\,\ab{i'}{i}\ab{v}{\log \log a_i} \Bigl(v = a_{i'} \rightarrow \\
\rightarrow \ab{j}{b'}\bigwedge_{\gn{\delta}< \epsilon'}\bigl(\delta(v,j,\overline r,  \overline R, F)  \leftrightarrow \tuple{i',j,\gn{\delta}} \in_\Ack j^*\bigr) \Bigr)
\end{multline}
for each large enough $i \in I$, as witnessed by $F\defeq R^* \cap [0,\log {a_{i}})$. (The reason for the restriction to large $i \in I $ is that we want all terms in a $\Delta_0$ formula $\delta$ with $\gn{\delta} < \epsilon'$ to evaluate to a number below $\log a_i$ on arguments below $\max(\log\log a_i, b', \leszek{\max(\overline r)})$. In this way, the value of $\delta$ on such arguments is unchanged when we replace $R^*$ by $R^* \cap [0,\log {a_{i}})$.)

Arguing as in the first-order case, one can check that for a fixed $i \in I$ the statement (\ref{eqn:char-of-R^*})
is equivalent to a $\Delta_0$ formula with G\"odel number below $\epsilon$ and parameters
$b', j^*, a_i, \overline r, \overline R$, where the tuple $\tuple{b', j^*}$ is
below $b$. Therefore, we can apply \eqref{cond:hash} from the inductive hypothesis to conclude that $(M,\Y)$ satisfies
\begin{multline}\label{eqn:char-of-S^*}
\exists F  \! \subseteq \! [0, \log a_{i})\,\ab{i'}{i}\ab{v}{\log \log a_i} \Bigl(v = a_{i'} \rightarrow \\
\rightarrow \ab{j}{b'}\bigwedge_{\gn{\delta}< \epsilon'}\bigl(\delta(v,j,\overline s,  \overline S, F)  \leftrightarrow \tuple{i',j,\gn{\delta}} \in_\Ack j^*\bigr) \Bigr)
\end{multline}
for each large enough $i \in I$. \leszek{Note that (\ref{eqn:char-of-S^*}) is} 
(\ref{eqn:char-of-R^*}) with $\overline r, \overline R$ replaced by $\overline s, \overline S$.

Let $T \in \Y$ be the tree consisting of all 0--1 strings $\sigma$ such that, for each $i \in I$
large enough (the largeness is expressed by a single inequality, with no quantifiers involved) and such that $a_{i} < |\sigma|$, the string $\sigma{\upharpoonright}_{\log a_i}$ is the characteristic function of $F \subseteq [0,\log a_i)$ satisfying (\ref{eqn:char-of-S^*}) for $i$.
By the previous paragraph, there are arbitrarily large elements of $T$, so by $\WKL$ in $(M, \Y)$,
there is an infinite path $B \in \Y$ through $T$. Any such path is the characteristic function of a set that we can use as $S^*$.
\end{proof}

\begin{remark}
A more refined version of the proof of Theorem \ref{thm:wkl-iso},
in which the exponentially closed cut used to lower-bound the parameter $b$
is decoupled from the $\Sigma^0_1$-cut indexing a cofinal subset of the first-order universe,
can be used to show the following. If $(M,A)$ is a countable model of $\bd \Sigma_1(A) + {\exp} + \neg \ind \Sigma_1(A)$,
then any exponentially closed cut in $M$ that contains some $\Sigma_1(A)$-definable cut
is the greatest initial segment fixed pointwise by some automorphism of $(M,A)$.
This is also the information about greatest pointwise fixed initial segments
that can be obtained from the previously published arguments drawing on \cite{smorynski}
such as those of \cite{kossak:extensions}/\cite{kossak:extensions-correction} or \cite{kaye:properties}.
\end{remark}

We record the following immediate consequence of Theorem \ref{thm:wkl-iso}.
Further implications of the theorem for $\WKL^*_0$ are discussed in the next section.

\begin{corollary}\label{cor:wkl-elem-equiv}
Let $(M, \X)$, $(M, \Y)$ be models of $\WKL^*_0$
such that $(M,\X \cap \Y) \models \neg \ind \Sigma^0_1$. Let $\overline{c}$ be a tuple of elements of $M$ and $\overline{C}$ be a tuple of elements of $\X \cap \Y$.

Then $(M, \X, \overline{c}, \overline{C}) \equiv (M, \Y, \overline{c}, \overline{C})$,
and if $\X \subseteq \Y$ then $(M,\X) \preccurlyeq (M,\Y)$,
where the elementarity applies to all $\mathcal{L}_2$-formulas.
\end{corollary}

\section{Consequences for $\WKL^*_0$} \label{sec:wklstar}

In this section, we show how Theorem \ref{thm:wkl-iso} implies that
Weak K\"onig's Lemma combined with the failure of $\Sigma^0_1$-induction
has a number of unusual properties. In particular, we prove that in $\WKL^*_0 + \neg \ind \Sigma^0_1$ the analytic hierarchy collapses (Corollary \ref{cor:analytic-hierarchy-collapses}, Theorem \ref{thm:analytic-to-arithmetical}) and the low basis theorem fails in a very strong sense (Theorem \ref{thm:no-low-basis}). We also show that $\WKL$ is the strongest $\Pi^1_2$ statement that is $\Pi^1_1$-conservative over $\RCA^*_0 + \neg \ind \Sigma^0_1$ (Theorem \ref{thm:pi11-cons-over-not-is1}).

\leszek{Most of the main results of this section can be interpreted 
in general model-theoretic terms, cf.~e.g.~\cite[Chapter 2.2]{tent-ziegler}.
If we view 
arithmetical formulas as quantifier-free, $\Pi^1_1$ as purely universal, $\Pi^1_2$ as $\forall \exists$, and so forth, 
then, for instance, Corollary~\ref{cor:analytic-hierarchy-collapses}
says that the $\forall\exists $~theory $\WKL_0^*+\neg\ind\Sigma^0_1$
is model complete. Together with the $\Pi^1_1$~conservativity of $\WKL_0^*+\neg\ind\Sigma^0_1$
over $\RCA^*_0+\neg\ind\Sigma^0_1$ from Simpson--Smith~\cite{simpson-smith},
this implies that $\WKL_0^*+\neg\ind\Sigma^0_1$ is the model companion
of $\RCA^*_0+\neg\ind\Sigma^0_1$ (by Theorem \ref{thm:analytic-to-arithmetical},
it actually has the slightly stronger property of being a model completion of $\RCA^*_0+\neg\ind\Sigma^0_1$).
Thus, models of $\WKL_0^*+\neg\ind\Sigma^0_1$ are precisely
the existentially closed models of $\RCA^*_0+\neg\ind\Sigma^0_1$,
and $\WKL_0^*+\neg\ind\Sigma^0_1$ is the strongest $\forall\exists$-theory that has the same
purely universal consequences as $\RCA^*_0+\neg\ind\Sigma^0_1$
(this is essentially what Theorem~\ref{thm:pi11-cons-over-not-is1} tells us).}

\leszek{In fact, most of the results mentioned above could be proved using Corollary~\ref{cor:wkl-elem-equiv}
and general model-theoretic arguments. 
However, we give proofs that are more specialized to arithmetic, because they
provide some additional information.}
We begin by checking that $\Sigma^0_1$-induction is not needed to justify the well-known fact that a model of $\WKL$ contains many coded submodels of $\WKL$ \cite[Corollary VIII.2.7]{simpson:sosoa}.

\begin{definition}
Given $(M,\X) \models \RCA^*_0$, a set $W \in \X$, and an element $k \in M$, let $W_k$ stand for $\{j \in M: \tuple{k,j} \in W\}$. We say that $W$ \emph{codes} the family of sets $\{W_k: k\in M\}$, and we refer to $(M, \{W_k: k\in M\})$ as a \emph{coded $\bbomega$-submodel of $(M,\X)$}.
We use the abbreviated term \emph{coded $\bbomega$-model} if $(M, \X)$ is clear from the context or irrelevant, and also to refer to $(\N, \{W_k: k\in \N\})$ when working in a formal theory.
\end{definition}

\begin{lemma}\label{lem:wkl-coded}
Let $(M, \X) \models \RCA^*_0$ and let $A \in \X$. There exists a $\Delta_1(A)$-definable infinite 0--1 tree $T \in \X$ such that if $W \subseteq M$ is an infinite path in $T$, then $(M, \{W_k: k\in M\})$ is a model of $\WKL^*_0$ with $A = W_0$.

As a consequence, for every $(M,\X) \models \WKL^*_0$ and every $A \in \X$,
there exists $W \in \X$ coding an $\bbomega$-model $(M,\W)$ of $\WKL^*_0$ with $A \in \W$.
\end{lemma}

Note that it is not assumed in the lemma that $W \in \X$ or even that $\bd\Sigma_1(W)$ holds:
an arbitrary subset $W \subseteq M$ with the property that the characteristic function of $W{\upharpoonright}_k$ is a node of $T$
for every $k \in M$ will have the property that $(M, \{W_k: k\in M\}) \models \WKL^*_0$.

\begin{proof}
Let $(M, \X) \models \RCA^*_0$ and $A \in \X$ be given.

It is easy to check that the well-known equivalence between $\WKL$
and the $\Sigma^0_1$-separation principle ~\cite[Lemma~IV.4.4]{simpson:sosoa} holds over $\RCA^*_0$.
Moreover, $\Sigma^0_1$-separation clearly implies $\Delta^0_1$-comprehension.
Thus, given $W\subseteq M$,
to prove that the structure $(M, \{ W_k: k\in M\})$ is a model of $\WKL^*_0$ it suffices to show
that it satisfies $\ind \Delta^0_0$,
that $\{ W_k: k\in M\}$ is closed under join,
and that for any $k \in M$ and any two $\Sigma_1(X)$ formulas $\varphi(x,X), \psi(x,X)$
(possibly with first-order parameters from $M$),
if $\varphi(\cdot, W_k)$ and $\psi(\cdot, W_k)$ define disjoint sets,
then there is some $ W_\ell$ separating them.
Given an arbitrary tree from a model $(M,\X) \models \RCA^*_0$
and an arbitrary $W \subseteq M$ that is a path in that tree,
$\ind \Delta_0(W)$ will always hold,
and a fortiori $\{ W_k: k\in M\}$ will satisfy $\ind \Delta^0_0$,
simply because bounded initial segments of $W$ will always have the form $\Ack(s)$ for some $s \in M$.
So, it is enough to verify closure under join and $\Sigma^0_1$-separation.

Let $(\varphi_i(x,X))_{i \in M}$ be an effective listing of all $\Sigma^0_1$ formulas in $M$
with a unique first-order variable $x$ and a unique second-order free variable $X$. (The formulas may involve numerals that represent first-order parameters.)
We take $T$ to be a 0--1 tree describing finite approximations to some $W$ such that:
\begin{enumerate}[(a)]
\item\label{cond:W:0} $ W_0 = A$,
\item\label{cond:W:join} $ W_{\tuple{0,i,j}} =  W_i \oplus  W_j$ for every $i,j$,
\item\label{cond:W:sep} for every $i,j,k\in M$, if the sets defined by $\varphi_i(\cdot, W_k)$ and $\varphi_j(\cdot, W_k)$ are disjoint,
then $ W_{\tuple{1,i,j,k}}$ is a separating set for them.
\end{enumerate}
Formally, to determine whether a finite binary string $\sigma$ belongs to $T$,
we look at the largest $s$ such that $\tuple{i,j} < |\sigma|$ for each $i,j < s$,
so that $\sigma$ can be viewed as containing an $s \times s$ binary-valued matrix.
The conditions \ref{cond:W:0}--\ref{cond:W:sep} are then expressed as requirements concerning this matrix.
For instance, \ref{cond:W:0} is expressed by requiring that $\sigma(\tuple{0,i}) = 1$ iff $i \in A$, for each $i < s$.
We leave \ref{cond:W:join} to the reader.
The condition \ref{cond:W:sep} is expressed by requiring that,
whenever $\tuple{1,i,j,k} < s$ and there is no $x < s$ such that there are witnesses below $s$
for both the formulas
\begin{align*}
\psi_i(x) &\quad\defeq\quad  \varphi_i(x,\{w < s: \sigma(\tuple{k,w}) = 1\}),\\
\psi_j(x) &\quad\defeq\quad  \varphi_j(x,\{w<s: \sigma(\tuple{k,w}) = 1\}),
\end{align*}
then for each $x < s$,
the existence of $y < s$ witnessing $\psi_i(x)$
implies that $\sigma(\tuple{\tuple{1,i,j,k},x}) = 1$,
and
the existence of $y < s$ witnessing $\psi_j(x)$
implies that $\sigma(\tuple{\tuple{1,i,j,k},x}) = 0$.

It is straightforward to verify that $T$ is an infinite $\Delta_1(A)$-definable binary tree.
By our discussion above, if $ W\subseteq M$ is an infinite path in~$T$, 
then $ W$ codes an $\bbomega$-model satisfying $\WKL^*_0$ and containing $A$.
\end{proof}

\begin{remark}
The construction of a single infinite binary tree whose paths correspond to coded models of $\WKL$ has been used earlier,
for instance in the context of comparing variants of Weihrauch reducibility \cite[Proposition 4.9]{hj:reduction-notions}. It can also be applied to give a relatively simple proof of the result of \cite{hajek:interpretability, avigad:formalizing-forcing} that $\WKL_0$ does not have superpolynomial proof speedup over $\RCA_0$ for $\Pi^1_1$ sentences
(see~\cite{avigad:logic-book}, or see~\cite{wong:interpreting-wkl-act} for a very similar argument expressed in terms of the arithmetized completeness theorem). Combined
with the formalized forcing argument described at the end of the present section,
it can also prove the fact, alluded to in \cite{kwy:ramsey-proof-size}, that a similar non-speedup result
holds also for $\WKL^*_0$ over $\RCA^*_0$.
\end{remark}

Note that for any $\mathcal{L}_2$-formula $\varphi(\overline x, \overline X)$, there is a single arithmetical
formula in variables $\overline x, \overline X, W$ that expresses the property ``$W$ codes an $\bbomega$-model containing $\overline X$ and
satisfying $\varphi(\overline x, \overline X)$'' in $\RCA^*_0$.
The number of first-order quantifier alternations in this formula will depend on $\varphi$.

\begin{corollary}\label{cor:analytic-hierarchy-collapses}
Over $\WKL^*_0 + \neg \ind \Sigma^0_1$,
any $\mathcal{L}_2$-formula $\varphi(\overline x, \overline X)$
is provably equivalent both to a $\Sigma^1_1$ formula and to a $\Pi^1_1$ formula.

More specifically, $\WKL^*_0 + \neg \ind \Sigma^0_1$ proves that the following
three statements are equivalent for any $\overline x, \overline X$:
\begin{enumerate}[(i)]
\item $\varphi(\overline x, \overline X)$,
\item ``there exists a coded $\bbomega$-model of $\WKL^*_0 + \neg \ind \Sigma^0_1 + \varphi(\overline x, \overline X)$'',
\item ``every coded $\bbomega$-model of $\WKL^*_0 + \neg \ind \Sigma^0_1$ containing $\overline x, \overline X$
satisfies $\varphi(\overline x, \overline X)$''.
\end{enumerate}
\end{corollary}
\begin{proof}
By Lemma \ref{lem:wkl-coded}, it is provable in $\WKL^*_0 + \neg \ind \Sigma^0_1$ that for every $\overline x, \overline X$  there is a coded $\bbomega$-model of $\WKL^*_0 + \neg \ind \Sigma^0_1$ containing $\overline X$. By Corollary \ref{cor:wkl-elem-equiv}, it is provable in $\WKL^*_0$ that any such model satisfies $\varphi(\overline x, \overline X)$ if and only if $\varphi(\overline x, \overline X)$ holds. This proves both (i)~$\leftrightarrow$~(ii) and (i)~$\leftrightarrow$~(iii).
\end{proof}

Corollary \ref{cor:analytic-hierarchy-collapses} says that in $\WKL^*_0 + \neg \ind \Sigma^0_1$
the analytic hierarchy collapses to $\Delta^1_1$. At the end of this section, we will explain
how this collapse result can be strengthened.

We can use Corollary \ref{cor:analytic-hierarchy-collapses} and other consequences of Theorem \ref{thm:wkl-iso} to give a characterization of $\Pi^1_1$-conservativity over $\RCA^*_0 + \neg \ind \Sigma^0_1$ for $\Pi^1_2$ sentences.
Before that, however, we digress briefly in order to point out that the $\Pi^1_1$ consequences of $\RCA^*_0 + \neg \ind \Sigma^0_1$ are not as pathological a theory as might be supposed.

\begin{proposition}\label{prop:neg-is1-aca0}
The $\Pi^1_1$ consequences of $\RCA^*_0 + \neg \ind \Sigma^0_1$ are contained in those of $\ACA_0$; in particular,
they are true in all $\omega$-models. They are incomparable to the $\Pi^1_1$ consequences of $\RCA_0 + \ind \Sigma^0_n$, for any $n \in \omega$.
\end{proposition}

\begin{proof}
As discussed in Section \ref{sec:towsner}, the principle $\mathrm{C}\Sigma^0_{n+1}$ is a $\Pi^1_1$ statement provable in $\RCA^*_0 + \neg \ind \Sigma^0_1$ but not in $\RCA_0 + \ind \Sigma^0_n$. Of course, $\mathrm{Con}(\ind \Delta_0 + \exp)$ or the totality of the iterated exponential function are examples of $\Pi^1_1$ statements provable in $\RCA_0$ but not in $\RCA^*_0 + \neg \ind \Sigma^0_1$.

It remains to prove that the $\Pi^1_1$ consequences of $\RCA^*_0 + \neg \ind \Sigma^0_1$ are contained in those of $\ACA_0$, or by contraposition, that any $\Sigma^1_1$ statement consistent with $\ACA_0$ is consistent with
$\RCA^*_0 + \neg \ind \Sigma^0_1$. This is implicit in the proof of \cite[Proposition 2.2]{kossak:omega}.
We spell out the argument. Let $(M,\X)$ be a countable recursively saturated model of $\ACA_0 + \exists X \alpha(X)$, where $\alpha$ is arithmetical, and let $A \in \X$ be such that $(M, A) \models \alpha(A)$. The one-sorted structure $(M,A)$ is a countable recursively saturated model of $\PA(A)$.

By recursive saturation, for every $a \in M$ there exists $b \in M$ that is above any element definable in $(M, A)$ with parameters below $a$. By another application of recursive saturation, there exists a sequence $\tuple{a_k}_{k < c}$ for $c$ nonstandard such that each $a_k$ is above any element definable in $(M, A)$ with parameters below $a_{k-1}$. Let $I$ be the initial segment of $M$ generated by $C = \{a_n: n \in \omega\}$, and let $A_I = A \cap I$. Note that $I$ is obviously closed under $\exp$ and that $A_I \oplus C$ is the intersection with $I$ of a definable set, so the structure $(I, A_I \oplus C)$ satisfies $\bd \Sigma^0_1$. Of course, it does not satisfy $\ind \Sigma^0_1$, because $\omega$ is $\Sigma_1(C)$-definable.

Since $I$ is closed under definability in $(M,A)$ and $\PA(A)$ has definable Skolem functions, $(I,A_I)$ is an elementary substructure of $(M,A)$. Thus, $(I, A_I) \models \alpha(A_I)$. Putting things together, we conclude that $(I, \Delta^0_1\textrm{-}\mathrm{Def}(I, A_I \oplus C))$ is a model of
$\RCA^*_0 + \neg \ind \Sigma^0_1 + \exists X \alpha(X)$.
\end{proof}


\begin{lemma}\label{lem:criterion-pi11-cons-over-not-is1}
Let $\psi$ be a $\Pi^1_2$ sentence.
Then $\psi$ is $\Pi^1_1$-conservative over $\RCA^*_0 +  {\neg \ind \Sigma^0_1}$  if and only if
$\RCA^*_0$ proves the $\Pi^1_1$ sentence
``every coded $\bbomega$-model of $\WKL^*_0 + \neg \ind \Sigma^0_1$ satisfies $\psi$''.
\end{lemma}

\begin{proof}
If $\psi$ is a $\Pi^1_2$ sentence which is $\Pi^1_1$-conservative over $\RCA^*_0 +  {\neg \ind \Sigma^0_1}$,
then, by \cite{simpson-smith} and a routine union of chains argument (cf.~\cite{yokoyama:conservativity}),
$\WKL \land \psi$ is also $\Pi^1_1$-conservative over $\RCA^*_0 +  {\neg \ind \Sigma^0_1}$.
By Corollary \ref{cor:analytic-hierarchy-collapses}, $\WKL^*_0 + \psi + {\neg \ind \Sigma^0_1}$
proves that every coded $\bbomega$-model of $\WKL^*_0 + \neg \ind \Sigma^0_1$ satisfies $\psi$.
So, by $\Pi^1_1$-conservativity, also $\RCA^*_0 +  {\neg \ind \Sigma^0_1}$ proves this statement, which is then
obviously provable in $\RCA^*_0$ since $\RCA^*_0 +  {\ind \Sigma^0_1}$ rules out the existence
of coded $\bbomega$-models of $\neg \ind \Sigma^0_1$.

On the other hand, if there is a $\Sigma^1_1$ sentence $\xi$ consistent with $\RCA^*_0 +  {\neg \ind \Sigma^0_1}$
but not with $\RCA^*_0 +  {\neg \ind \Sigma^0_1}+\psi$, then by \cite{simpson-smith} there is a model of $\WKL^*_0 + \neg \ind \Sigma^0_1 + \xi$.
Clearly, this model must also satisfy $\neg \psi$, so by Corollary \ref{cor:analytic-hierarchy-collapses}
it contains a coded $\bbomega$-submodel satisfying $\WKL^*_0 + \neg \ind \Sigma^0_1 + \neg\psi$.
\end{proof}

\begin{remark}
Even though Corollary \ref{cor:analytic-hierarchy-collapses} applies to arbitrary $\mathcal{L}_2$-formulas $\varphi$,
the proof of Lemma \ref{lem:criterion-pi11-cons-over-not-is1} employs a union of chains argument that only works for $\Pi^1_2$ statements. In fact, Lemma \ref{lem:criterion-pi11-cons-over-not-is1} cannot be generalized to $\Sigma^1_2$ sentences $\psi$, as witnessed by $\psi \defeq \neg \WKL$~\cite[Proposition~11]{ky:categorical}.
\end{remark}

Lemma \ref{lem:criterion-pi11-cons-over-not-is1} already implies that the set of $\Pi^1_2$ sentences that are $\Pi^1_1$-conservative over $\RCA^*_0 \leszek{+ \neg \ind \Sigma^0_1}$ is \leszek{c.e.}, 
and thus \leszek{it is} a computably axiomatized theory. The next result states that in fact this theory  has a rather simple axiomatization.

\begin{theorem}\label{thm:pi11-cons-over-not-is1}
Let $\psi$ be a $\Pi^1_2$ sentence. Then $\psi$ is $\Pi^1_1$-conservative over $\RCA^*_0 +  {\neg \ind \Sigma^0_1}$  if and only if $\WKL^*_0 + \neg \ind \Sigma^0_1 \vdash \psi$.
\end{theorem}

\begin{proof}
Of course, if the $\Pi^1_2$ sentence $\psi$ follows from $\WKL^*_0 + \neg\ind\Sigma^0_1$,
then it is $\Pi^1_1$-conservative over $\RCA^*_0 + \neg\ind\Sigma^0_1$, by \cite{simpson-smith}.

On the other hand, if $\psi$ is $\Pi^1_2$ and $\Pi^1_1$-conservative over $\RCA^*_0 + \neg \ind\Sigma^0_1$,
then, by Lemma \ref{lem:criterion-pi11-cons-over-not-is1}, $\RCA^*_0$ proves that
every coded $\bbomega$-model of $\WKL^*_0 + \neg \ind \Sigma^0_1$ satisfies $\psi$.
But then Corollary \ref{cor:analytic-hierarchy-collapses} implies that $\WKL^*_0 + \neg\ind\Sigma^0_1$ proves $\psi$.
\end{proof}

\begin{remark}
Theorem \ref{thm:pi11-cons-over-not-is1} means in particular that,
in contrast to the tree forcing used to add paths through binary trees to a model,
any forcing notion that can be used to obtain witnesses to statements that
do not follow from $\WKL + \neg \ind \Sigma^0_1$ will not in general preserve $\bd \Sigma^0_1$ when applied
to a ground model satisfying $\bd \Sigma^0_1 + {\exp} + {\neg \ind \Sigma^0_1}$.
Note that unprovability from $\WKL + \neg \ind \Sigma^0_1$ is known e.g.~for relatively weak
Ramsey-theoretic such as the cohesive Ramsey's Theorem $\mathrm{CRT}^2_2$ (cf.~\cite{fkk:weak-cousins}),
which is very easy to witness using either Cohen or Mathias forcing.

Similar observations specifically about Cohen forcing, but in the context of models of $\neg {\exp}$, were made in \cite{fernandes:baire-feasible} and \cite{fff:analysis-weak}. Other results showing that it can be impossible
to preserve collection while forcing include \cite[Theorem 1]{Simpson-BCT-in-RCAs}, \cite[Theorem 4.3]{HSS-AMT}.
\end{remark}

An important feature of $\WKL$ is the low basis theorem \cite{jockusch-soare:low-basis}: if $T$ is any infinite 0--1 tree,
then $T$ has an infinite path $W$ which is \emph{low} in $T$, i.e.~every $\Delta_2(T \oplus W)$-definable
set is $\Delta_2(T)$-definable. It is known that the low basis theorem holds provably in $\RCA_0$,
in the sense that if $(M, \Delta^{0}_1\textrm{-}\mathrm{Def}(M,T)) \models \RCA_0 \land \textrm{``}T \textrm{ is an infinite 0--1 tree''}$, then there is some $W \subseteq M$ low in $T$ such that $W$ is an infinite path in $T$ and
$(M, \Delta^{0}_1\textrm{-}\mathrm{Def}(M,T\oplus W))$ still satisfies $\RCA_0$ (\cite{hajek-kucera:recursion-is1}, see \cite[Chapter I.3(b)]{hajek-pudlak}).

Chong and Yang \cite{chong-yang:jump-cut} showed that some important properties of low sets fail in the absence of $\Sigma^0_1$-induction: for example, over $\RCA^*_0 + \neg \ind \Sigma_1$ there is no non-computable low $\Sigma_1$-set.
Here, we prove that the low basis theorem fails in $\RCA^*_0$ in a strong way:
in general, a computable tree in a model of $\RCA^*_0$ will not even have an arithmetical path.

\begin{lemma}\label{lem:w-not-in-w}
Let $(M, \X) \models \RCA^*_0$, and let $W \in \X$ code an $\bbomega$-model $(M, \W) \models \RCA^*_0$.
Then $W \notin \W$.
\end{lemma}
\begin{proof}
If $W$ codes $\W$, then the sequence $\tuple{W_k}_{k \in M}$ contains all sets that belong to $\W$.
By a standard diagonalization argument, such a sequence cannot itself belong to $\W$.
\end{proof}

\begin{lemma} \label{lem:non-arith}
Let $(M,\X) \models \WKL^*_0 + \neg \ind \Sigma^0_1$. Then for every $A$ in $\X$, there is a set $B \in \X$
which is not arithmetically definable in $A$.
\end{lemma}
\begin{proof}
Suppose otherwise. By Lemma \ref{lem:wkl-coded} there exists $W \in \X$ coding an $\bbomega$-model $\W$ of $\WKL^*_0 + \neg \ind \Sigma^0_1$ such that $A \in \W$. Assume that there is an arithmetical formula $\varphi(x,A)$, with no set parameters other than $A$, which defines $W$ in $(M, A)$. Then $(M, \X) \models \exists X\, \forall x \,(x \in X \leftrightarrow \varphi(x,A))$.
On the other hand, Lemma \ref{lem:w-not-in-w} implies that $W \notin \W$, so $(M, \W) \not \models \exists X\, \forall x \,(x \in X \leftrightarrow \varphi(x,A))$. This contradicts Corollary \ref{cor:wkl-elem-equiv}.
\end{proof}

\begin{theorem} \label{thm:no-low-basis}
Let $(M,\X) \models \RCA^*_0$ and let $A \in \X$ be such that $(M, A) \models \neg \ind\Sigma_1(A)$.
Then there exists a $\Delta_1(A)$-definable infinite 0--1 tree $T$ such that for any
$W \subseteq M$, if $W$ is an infinite path in $T$, then $W$ is not arithmetically definable in $A$
(in particular, it is not low in $A$).
\end{theorem}
\begin{proof}
Let $(M, \X)$ and $A$ be as above, and let $T$ be the $\Delta_1(A)$-definable tree from Lemma \ref{lem:wkl-coded}.
If $W$ is an infinite path in $T$, then, letting $\W := \{W_k: k \in M\}$, we have
$(M,\W) \models \WKL^*_0$ and $A \in \W$.
If $W$ were arithmetically definable in $A$, then every set $B \in \W$ would also be arithmetically
definable in $A$, contradicting Lemma \ref{lem:non-arith}.	
\end{proof}

We may reformulate this as a reverse mathematics-style statement, which follows immediately from
Theorem \ref{thm:no-low-basis} and the aforementioned provability of the low basis theorem in $\RCA_0$.

\begin{corollary} For each $n \ge 2$, the following statements are equivalent provably in $\RCA^*_0$:
\begin{enumerate}[(i)]
\item $\RCA_0$,
\item for every infinite 0--1 tree $T$, there exists a $\Sigma_n(T)$-set $\extset W$ such that
$\extset W$ is an infinite path in $T$ and $\ind\Sigma_1(\extset W)$ holds,
\item for every infinite 0--1 tree $T$, there exists a $\Sigma_n(T)$-set $\extset W$ such that
$\extset W$ is an infinite path in $T$.
\end{enumerate}
\end{corollary}

As far as we know, the following purely model-theoretic question remains open. If the answer is negative,
then it has to be witnessed by a model that is not only uncountable, but does not have countable cofinality.

\begin{question}
Is it the case that for every model $(M,\X) \models \RCA^*_0$ and every infinite 0--1 tree $T \in \X$,
there exists an infinite path in $T$?
\end{question}

To conclude this section, we return to the topic of the collapse of the analytic hierarchy.
We discuss how to strengthen Corollary \ref{cor:analytic-hierarchy-collapses}
to say that provably in $\RCA^*_0$, if $A$ is any set witnessing the failure of $\ind \Sigma^0_1$,
then each $\mathcal{L}_2$-formula is equivalent not merely to a $\Delta^1_1$ formula, but
to an arithmetical formula with $A$ as parameter.
In particular, if even $\ind \Sigma_1$ fails, then every $\mathcal{L}_2$-formula is equivalent
to an arithmetical one with no additional parameters.
The strengthened collapse result is
proved by a forcing argument which is routine but tedious, so we only provide a sketch.

It is well-known that if $T$ is an infinite 0--1 tree in a countable model of $\RCA^*_0$, then forcing
with $\Delta_1(T)$-definable infinite subtrees of $T$ gives rise to an infinite path in $T$
that still satisfies $\bd\Sigma^0_1$ \cite{simpson-smith}. The analogous forcing construction for adding paths to trees
that live in models of $\RCA_0$
was formalized within $\RCA_0$ in \cite[Sections 4--5]{avigad:formalizing-forcing}. Here, we formalize in $\RCA^*_0$
a variant of this forcing such that the generic sets correspond to coded $\bbomega$-models of $\WKL^*_0$.

The following definition is made in $\RCA^*_0$.

\begin{definition}
Given any set $X$ and another set $A$ that we treat as a parameter, let $T_{X,A}$ be a tree
defined like the one in Lemma \ref{lem:wkl-coded} but with paths corresponding
to sets $W$ that code an $\bbomega$-model of $\WKL^*_0$ with $W_0 = X$, $W_1 = A$.
The \emph{conditions} of the forcing notion $\mathbb{P}_{X,A}$ are the $\Delta_1(X,A)$-definable
infinite subtrees of $T_{X,A}$, ordered by inclusion.
A \emph{first-order name} has the form $\dot k$ for a natural number $k$,
and it is intended to denote $k$ itself. A \emph{second-order name} has the form $\dot G_k$
for a natural number $k$, and it is intended to denote $G_k$ where $G$ is a generic
path in $T_{X,A}$.

For a condition $S$ and a \emph{sentence} $\varphi$ of the language obtained
by extending $\mathcal{L}_2$ with all first- and second-order names, treated as constants,
the forcing relation $S \Vdash_{X,A} \varphi$ (with the subscripts often omitted below)
is defined in the following way.
If $\varphi$ is a purely first-order atom, then $S \Vdash \varphi$ if and only if $\varphi$ holds
(under the intended interpretation of names). If $\varphi$ has the form $t(\overline{\dot{k}}) \in \dot{G}_\ell$,
then $S \Vdash \varphi$ if the set $\{\sigma \in S: \sigma(\tuple{\ell, t(\bar k)}) = 0 \}$ is finite.
The $\Vdash$ relation is extended to non-atomic formulas (which we take to be built using $\neg, \land$, and $\exists$)
as follows:
\begin{align*}
S \Vdash \neg \varphi  &\quad\defeq\quad  \forall Q \! \preceq \! S\,(Q \not \Vdash \varphi),\\
S \Vdash \varphi \land \psi &\quad\defeq\quad  (S \Vdash \varphi \land S \Vdash \psi),\\
S \Vdash \exists x\, \varphi &\quad\defeq\quad  \forall Q \! \preceq \! S\,\exists R \! \preceq \! Q\,\exists k\, (R \Vdash \varphi(\dot k)),\\
S \Vdash \exists X\, \varphi &\quad\defeq\quad  \forall Q \! \preceq \! S\,\exists R \! \preceq \! Q\,\exists k\, (R \Vdash \varphi(\dot G_k)).
\end{align*}
\end{definition}

Note that for any $\mathcal{L}_2$-formula $\varphi(x_1,\ldots,x_n, Y_1, \ldots, Y_m)$, the statement $S \Vdash_{X,A} \varphi(\overline{\dot x}, \dot G_{y_1}, \ldots, \dot G_{y_m})$
can be expressed by an arithmetical formula in $\overline x, \overline y, S, X,$ and $A$.

For a model $(M, \X) \models \RCA^*_0$ and $X,A \in \X$, a generic filter $G$ in $\mathbb{P}_{X,A}$ can be identified
with an infinite path in $T_{X,A}$, which codes an $\bbomega$-model $(M,\{G_k: k \in M\}$ of $\WKL^*_0$
with $G_0 = X$ and $G_1 = A$. We refer to this structure as $M[G]$.
One can prove the following two statements by (simultaneous) induction on the complexity of $\varphi$.
Firstly, if $S$ is a condition, then $S \Vdash \varphi$ iff for every generic $G$ with $S \in G$ it holds that $M[G] \models \varphi$
(as usual, under the intended interpretation of names). Secondly, if $G$ is generic and $M[G] \models \varphi$,
then there is some condition $S \in G$ such that $S \Vdash \varphi$.

If $(M, \X)$ is a model of $\WKL^*_0$ and $A$ is a witness for the failure of $\ind \Sigma^0_1$ in $(M,\X)$, then by Corollary \ref{cor:wkl-elem-equiv}, for any $\overline c \in M$, any $X \in \X$, and $G$ generic for $\mathbb{P}_{X,A}$, the structure $(M[G], \overline c, X)$ is elementarily equivalent to $(M, \X, \overline c, X)$. This means that either every condition will force $\varphi(\overline{\dot c}, \dot G_0)$ or every condition will force $\neg \varphi(\overline{\dot c}, \dot G_0)$, depending on whether $\varphi(\overline c, X)$ holds. Thus, we obtain the following collapse result:

\begin{theorem}\label{thm:analytic-to-arithmetical}
Let $\varphi(\overline x, X)$ be an $\mathcal{L}_2$-formula. Then $\WKL^*_0$ proves that following statements are equivalent for any $\overline x, X$ and any $Y$ such that $\neg \ind \Sigma_1(Y)$ holds:
\begin{enumerate}[(i)]
\item $\varphi(\overline x, X)$,
\item $\Vdash_{X,Y} \varphi(\overline{\dot x},\dot G_0)$.
\end{enumerate}
In particular, over $\WKL^*_0 + \neg \ind \Sigma_1$, $\varphi(\overline x, X)$ is equivalent to the arithmetical formula
$\Vdash_{X,\emptyset} \varphi(\overline{\dot x},\dot G_0)$, which has no free variables other than $X$.
\end{theorem}


\section{Generalization to higher levels}\label{sec:generalization}

If $n\ge 2$ and $(M, \X)$ is a model of $\RCA_0 + \bd \Sigma^0_n + \neg \ind \Sigma^0_n$,
then the $\Delta^0_n$-definable sets of $(M, \X)$ form a model of $\RCA^*_0 + \neg \ind \Sigma^0_1$.
Thus, it is reasonable to expect that the results of Sections \ref{sec:iso} and \ref{sec:wklstar}
say something about $\RCA_0 + \bd \Sigma^0_n + \neg \ind \Sigma^0_n$,
for instance about $\Pi^1_1$-conservativity over that theory.
The case of $n=2$ is particularly interesting,
as some prominent problems concerning the first-order consequences of Ramsey's Theorem for pairs
and related statements boil down to the question whether these statements
are $\Pi^1_1$-conservative over $\RCA_0 + \bd \Sigma^0_2 + \neg \ind \Sigma^0_2$.

In this section, we show that Theorem \ref{thm:wkl-iso} does indeed have some consequences
for higher levels of the arithmetic hierarchy. First, however, we generalize the results of
Proposition \ref{prop:neg-is1-aca0} to $n \ge 2$.

\begin{proposition}\label{prop:neg-isn-aca0}
For each $n \ge 1$, the $\Pi^1_1$ consequences of $\RCA^*_0 + \bd \Sigma^0_n + \neg \ind \Sigma^0_n$ are contained in those of $\ACA_0$ and incomparable to the $\Pi^1_1$ consequences of $\RCA_0 + \ind \Sigma^0_m$ for each $m \ge n$.
\end{proposition}

\begin{proof}
The argument for incomparability with the $\Pi^1_1$ consequences of $\RCA_0 + \ind \Sigma^0_m$ for $m \ge n$
is essentially the same as in the proof of Proposition \ref{prop:neg-is1-aca0}
and uses the statements $\mathrm{C}\Sigma^0_{m+1}$ and $\mathrm{Con}(\ind \Sigma_{n-1})$.

The argument showing that the $\Pi^1_1$ consequences of $\RCA^*_0 + \bd \Sigma^0_n + \neg \ind \Sigma^0_n$ are contained in those of $\ACA_0$ is also similar to the one in Proposition \ref{prop:neg-is1-aca0},
but now it has to be combined with a ``jump inversion'' argument based on \cite{belanger:coh}.
Let $(M,\X)$ be a countable recursively saturated model of $\ACA_0 + \alpha(A)$, where $A \in \X$ and $\alpha$ is arithmetical, and let $C, I, A_I$ be obtained as in the proof of Proposition \ref{prop:neg-is1-aca0}. We then know that
$(I, A_I \oplus C) \models \bd \Sigma^0_1 + \alpha(A_I)$ and $(I,  C) \models \neg \ind \Sigma^0_1$. Moreover,
$(I, A_I)$ is an elementary substructure of $(M,A)$, which means in particular that $(I, A_I)$ satisfies induction for all arithmetical formulas, and, by construction, that $(I, (A_I)^{(m)} \oplus C) \models \bd \Sigma^0_1$ for any $m$.

Let $C_1 := C$. If $n = 1$, we have nothing more to do.
Otherwise, note that $(I, (A_I)^{(n-2)}) \models \bd \Sigma^0_2$.
Since $(I, (A_I)^{(n-1)} \oplus C_1) \models \bd \Sigma^0_1$,
we can use Belanger's jump inversion Theorem \ref{thm:belanger-jump-inversion} to obtain some $C_2 \subseteq I$
such that $(I, (A_I)^{(n-2)} \oplus C_2) \models \bd \Sigma^0_2$ and $C_1$ is $\Delta_2(C_2)$-definable. If $n =2$, we are done, and otherwise, since $(I, (A_I)^{(n-3)}) \models \bd \Sigma^0_3$ and $(I, (A_I)^{(n-2)} \oplus C_2) \models \bd \Sigma^0_2$,
we can use Theorem~\ref{thm:belanger-jump-inversion} again to get $C_3$ such that
$(I, (A_I)^{(n-3)} \oplus C_3) \models \bd \Sigma^0_3$ and $C_1$ is $\Delta_2(\Delta_2(C_3))$-,
thus $\Delta_3(C_3)$-definable.
Continuing in this way, we eventually get $C_n$
such that $(I, A_I \oplus C_n) \models \bd \Sigma^0_n$ and $C_1$ is $\Delta_n(C_n)$-definable.
But then $(I, \Delta^0_1\textrm{-}\mathrm{Def}(A_I \oplus C_n))$
is a model of $\RCA^*_0 + \bd \Sigma^0_n + \neg \ind \Sigma^0_n + \exists X \alpha(X)$.
\end{proof}

\begin{corollary}
For any $n \ge 1$, the $\Pi^{1}_{1}$-consequences of $\RCA^*_0 + \bd \Sigma^0_n + \neg \ind \Sigma^0_n$ are contained in those of
\begin{equation}\label{eqn:weak-ib}
\RCA^*_0 + \bd \Sigma^0_n + \{\ind \Sigma^0_m \rightarrow \bd \Sigma^0_{m+1} : m \in \omega \}.
\end{equation}
\end{corollary}
The theory in (\ref{eqn:weak-ib}) is one of two natural relativizations of the theory IB from \cite{kaye:theory-kappa-like}.
In \cite{kky:ramsey-rca0star}, this is referred to as the ``weak'' relativization.

\begin{proof}
Let $(M,\X)$ be a model of the theory in (\ref{eqn:weak-ib}) and let $A \in \X$ satisfy $\alpha(A)$ with $\alpha$ arithmetical.
If $m \in \omega$ is the smallest such that $(M, \X) \not \models \ind \Sigma^0_{m}$, then ${m \ge n}$ and
$(M, \X) \models \bd \Sigma^0_{m}$,
so $(M, \Delta^0_{m-n+1}\textrm{-}\mathrm{Def}(M, \X)) \models \RCA^*_0 + \bd \Sigma^0_n + \neg \ind \Sigma^0_n + \exists X \alpha(X)$. On the other hand, if $(M,\X)$ satisfies induction for all arithmetical formulas,
then the closure of $\X$ under arithmetical definability witnesses that $\exists X \alpha(X)$ is consistent with $\ACA_0$ and thus,
by Proposition \ref{prop:neg-isn-aca0}, with $\RCA^*_0 + \bd \Sigma^0_n + \neg \ind \Sigma^0_n$.
\end{proof}

In order to generalize further results of Section \ref{sec:wklstar},
 in particular Lemma \ref{lem:criterion-pi11-cons-over-not-is1} and Theorem \ref{thm:pi11-cons-over-not-is1},
to arbitrary $n$, we need a suitable variant of Weak K\"onig's Lemma.

\begin{definition}
Let $\Delta^0_n$-$\WKL$ be the statement:
``for every $\Delta^0_n$-set $\extset T$ that is an infinite 0--1 tree,
there exists a $\Delta^0_n$-set $\extset W$ that is an infinite path in $\extset T$''.
\end{definition}

Note that $\Delta^0_1$-$\WKL$ is equivalent to $\WKL$ provably in $\RCA^*_0$.
On the other hand, Belanger \cite{belanger:coh} showed that $\Delta^0_2$-$\WKL$ is
equivalent to the cohesive set principle $\COH$ over $\RCA_0 + \bd \Sigma^0_2$.
The appearance of $\bd \Sigma^0_2$ here is not incidental. $\COH$,
being $\Pi^1_1$-conservative over $\RCA_0$ \cite{cholak-jockusch-slaman}, does not prove $\bd\Sigma^0_2$,
but an argument in the spirit of \cite[Proposition 5]{ferreira:feasible-analysis} shows that $\Delta^0_2$-$\WKL$ does.

\begin{proposition}\label{prop:n-wkl-bd}
For every $n \ge 1$, $\Delta^0_n$-$\WKL$ implies $\bd \Sigma^0_n$ over $\RCA^*_0$. In fact, $\bd \Sigma^0_n$ is already implied by the statement ``for every $\Delta^0_n$-set that is an infinite 0--1 tree and every number $k$, there is a node at level $k$ in the tree with infinitely many nodes above it''.
\end{proposition}

\begin{proof}
For fixed $n$, we argue that the statement above implies $\bd \Sigma^0_m$ by external induction on $m \le n$. The thesis obviously holds for $m = 1$. Now assume that it holds for $m < n$ and that $\bd \Sigma^0_{m+1}$ fails, and let $\psi(x,y)$ be a $\Pi^0_m$ formula such that for some $k$, it holds that $\forall \sigma \! \in \! \{0,1\}^k\, \exists y \, \psi(\sigma,y)$ but the witnesses $y$ cannot be bounded in a way independent of $\sigma$.

Consider the definable set of binary strings $\extset T$ consisting of all strings with length $\le k$ and all $\sigma^\frown \tau$ where $|\sigma| = k$ and there is no $y \le |\tau|$ with $\psi(\sigma, y)$. Then $\extset T$ is a $\Delta^0_{m+1}$-set (in fact, a $\Sigma^0_m$-set) by $\bd \Sigma^0_m$. Moreover, it is an infinite 0--1 tree, but for every $\sigma$ of length $k$ there are only finitely many vertices in $\extset T$ above $\sigma$.
\end{proof}

\begin{lemma}\label{lem:d0n-wkl-cons}
For every $n \ge 1$, $\Delta^0_n$-$\WKL$ is $\Pi^1_1$-conservative over $\RCA^*_0 + \bd\Sigma^0_n$.
Moreover, any countable topped model of $\RCA^*_0 + \bd\Sigma^0_n$
can be $\bbomega$-extended to a model of $\RCA^*_0 + \Delta^0_n$-$\WKL$.
\end{lemma}
\begin{proof}
For $n = 1$, this is \cite[Corollary 4.7]{simpson-smith}, and for $n = 2$,
it follows from results of \cite{csy:conservation-weaker} and \cite{belanger:coh}.

We prove the general case by induction on $n$.
Like in the proof of Proposition~\ref{prop:neg-isn-aca0}, we use a jump inversion argument based on Theorem \ref{thm:belanger-jump-inversion}.
Assume that the statement holds for $n$. To prove it for $n+1$, it is enough
to show that given any countable model $(M, A) \models \bd\Sigma^0_{n+1}$ and a $\Delta_{n+1}(A)$-definable
infinite 0--1 tree $\extset T$, we can find $B \subseteq M$ such that $(M,A \oplus B) \models \bd\Sigma^0_{n+1}$
and there is a $\Delta_{n+1}(B)$-definable path in $\extset T$.

Note that $(M, A') \models \bd\Sigma^0_{n}$ and $\extset T$ is $\Delta_n(A')$-definable.
So, by our inductive assumption,
there exists $\extset P \subseteq M$ such that $(M,A' \oplus \extset P) \models \bd\Sigma^0_{n}$
and there is a $\Delta_n(\extset P)$-definable path $\extset G$ in $\extset T$.
By Theorem \ref{thm:belanger-jump-inversion}, there is some $B \subseteq M$
such that $(M, A \oplus B) \models \bd \Sigma^0_{n+1}$ and $\extset P$ is $\Delta_2(B)$-definable.
Thus, $\extset G$ is $\Delta_n(\Delta_2(B))$-definable, and hence
$\Delta_{n+1}(B)$-definable because $\Delta_2(B)$-collection holds.
\end{proof}

It follows immediately from Theorem \ref{thm:wkl-iso} that any two countable
models of $\Delta^0_n$-$\WKL$ that $\bbomega$-extend
the same model of $\bd \Sigma^0_n + \neg \ind \Sigma^0_n$
are to some degree similar.

\begin{corollary}
Let $n \ge 1$, and let $(M, \X)$ and $(M,\W)$ be countable models of $\RCA^*_0 + \Delta^0_n$-$\WKL$
such that $(M,\X \cap \W) \models \neg \ind \Sigma^0_n$.
Let $\overline{c}$ be a tuple of elements of $M$ and $\overline{C}$ be a tuple of subsets of $M$ that are $\Delta^0_n$-definable in both $\X$ and $\W$. Then there exists an isomorphism $h$ between $(M, \Delta^0_n\textrm{-}\mathrm{Def}(M,\X))$ and $(M, \Delta^0_n\textrm{-}\mathrm{Def}(M,\W))$ such that $h(\overline{c})=\overline{c}$ and $h(\overline{C})=\overline{C}$.
\end{corollary}
\begin{proof}
Apply Theorem \ref{thm:wkl-iso} to $(M, \Delta^0_n\textrm{-}\mathrm{Def}(M,\X))$ and $(M, \Delta^0_n\textrm{-}\mathrm{Def}(M,\W))$.
\end{proof}

We introduce an auxiliary piece of notation.

\begin{definition}
Given $n\ge 0$ and a set $A$, we write $X\ll^{n}_{A} Y$ for the statement
``for every $\Delta_n(X\oplus A)$-set $\extset T$ that is an infinite 0--1 tree,
there exists a $\Delta_n({Y\oplus A})$-set $\extset W$ that is an infinite path in $\extset T$''.
\end{definition}
The choice of the $\ll$ symbol is inspired by the computability-theoretic notation
$X \ll Y$, which means that $Y$ has PA-degree relative to $X$, that is,
every $X$-computable infinite 0--1 tree has a $Y$-computable path.

\begin{lemma}\label{lem:ll-facts}
For each $n \ge 1$, $\RCA^{*}_{0}$ proves that:
\begin{enumerate}[(a)]
\item\label{pt:ll:B} for any sets $X,Y,A$, if $X \ll^{n}_{A} Y$, then $\bd \Sigma_{n}{(X\oplus A)}$ holds,
\item\label{pt:ll:cod} for any sets $X,Y,A$: $X \ll^{n}_{A} Y$ holds
if and only if
there exists a $\Delta_{n}({Y\oplus A})$-set that codes an $\bbomega$-model of $\WKL^*_0$ containing $(X\oplus A)^{\leszek{(}n-1\leszek{)}}$,
\item\label{pt:ll:WKL} $\Delta^0_{n}$-$\WKL$ is equivalent to $\forall X\,\forall Z\, \exists Y\, X \ll^{n}_{Z} Y$.
\end{enumerate}
\end{lemma}
\begin{proof}
To prove \ref{pt:ll:B}, note that if $X \ll^{n}_{A} Y$, then in particular for every number $k$ and every $\Delta_{n}(X\oplus A)$-set that is an infinite 0--1 tree, there is a node at level $k$ with infinitely many nodes above it. By the argument from the proof of Proposition~\ref{prop:n-wkl-bd}, this implies $\bd \Sigma_{n}{(X\oplus A)}$.

We turn to \ref{pt:ll:cod}. First assume that there is a $\Delta_{n}({Y\oplus A})$-set $\extset W$ coding an $\bbomega$-model $\extset \W$ of $\WKL^*_0$ such that $(X\oplus A)^{\leszek{(}n-1\leszek{)}} \in \extset \W$. Clearly, every element of $\extset \W$ is a $\Delta_{n}({Y\oplus A})$-set. Moreover, every $\Delta_n(X \oplus A)$-set belongs to $\extset \W$, because $\extset \W$ contains $(X\oplus A)^{\leszek{(}n-1\leszek{)}}$
and is closed under $\Delta^0_1$-comprehension.
If the $\Delta_n(X \oplus A)$-set happens to be an infinite 0--1 tree, it will have an infinite path belonging to $\leszek{\extset \W}$.
So, $X \ll^{n}_{A} Y$ holds. In the other direction, if $X \ll^{n}_{A} Y$ holds, then by \ref{pt:ll:B} the
$\Delta_{n}({X\oplus A})$-sets form a model of $\RCA^*_0$, so by Lemma~\ref{lem:wkl-coded} there is
a single infinite $\Delta_{n}({X\oplus A})$-definable 0--1 tree $\extset T$ such that any $\extset W$
that is a path in $\extset T$
codes an $\bbomega$-model of $\WKL^*_0$ containing $(X\oplus A)^{\leszek{(}n-1\leszek{)}}$.
Since  $X \ll^{n}_{A} Y$, some such $\extset W$ is a $\Delta_{n}({Y\oplus A})$-set.

In the proof of \ref{pt:ll:WKL}, the right-to-left direction is immediate. In the other direction, assuming $\Delta^0_{n}$-$\WKL$
we get $\bd \Sigma^0_n$ by Proposition \ref{prop:n-wkl-bd}. So, by Lemma~\ref{lem:wkl-coded} again, given sets $X$ and $Z$ there is a single infinite $\Delta_{n}({X \oplus Z})$-definable 0--1 tree $\extset T$ such that any path in $\extset T$
codes an $\bbomega$-model of $\WKL^*_0$ containing $(X\oplus Z)^{\leszek{(}n-1\leszek{)}}$. By $\Delta^0_{n}$-$\WKL$, there
is some $Y$ for which there is a $\Delta_n(Y)$-definable path in $\extset T$. Then $X \ll^{n}_{Z} Y$
holds for any such $Y$.
\end{proof}


\begin{definition}[$\ll^{n}_{A}$-basis theorem]
Let $n \ge 1$ and let $\psi$ be a $\Pi^1_2$ sentence of
the form $\forall X \, \exists Y \, \alpha(X,Y)$ where $\alpha$ is arithmetical.
For a given set $A$, the \emph{$\ll^{n}_{A}$-basis theorem for $\psi$} is the following statement:
\begin{quote}
 for any sets $Z$ and $X$, if $X\ll_{A}^{n}Z$ then there exists a $\Delta_{n}({Z\oplus A})$-set $\extset Y$ such that $\alpha(X,\extset Y)$ and $X\oplus \extset Y\ll_{A}^{n} Z$.
\end{quote}
\end{definition}


The following result can be viewed as a generalization of Lemma \ref{lem:criterion-pi11-cons-over-not-is1}.

\begin{theorem}\label{lem:criterion-pi11-cons-over-bd-not-ind}
Let $n \ge 1$ and let $\psi$ be a $\Pi^1_2$ sentence of
the form $\forall X \, \exists Y \, \alpha(X,Y)$ where $\alpha$ is arithmetical.
Then $\psi$ is $\Pi^1_1$-conservative over $\RCA^*_0 + \bd\Sigma^0_n + \neg \ind \Sigma^0_n$
if and only if $\RCA^*_0 + \bd \Sigma^0_n$ proves the following $\Pi^1_1$ sentence $\gamma^{n}_{\psi}$:
\begin{quote}
``for every set $A$, if $\ind \Sigma_{n}(A)$ does not hold, then the $\ll^{n}_{A}$-basis theorem for $\psi$ holds''.
\end{quote}
\end{theorem}
We will write $\gamma_\psi$ instead of $\gamma^n_\psi$ whenever $n$ is clear from the context,
including in the proof of Theorem \ref{lem:criterion-pi11-cons-over-bd-not-ind} itself.

\begin{proof}
The left-to-right direction can be proved by an argument that is similar to the one in the proof of Lemma \ref{lem:criterion-pi11-cons-over-not-is1}, but a bit more complicated. We argue in a slightly different way in order to obtain some
additional information (Corollary~\ref{cor:proves-gamma}).


Let $\psi$ be the $\Pi^1_2$ sentence $\forall X \, \exists Y \, \alpha(X,Y)$.
We argue that the $\Pi^1_1$ sentence $\gamma_\psi$ is provable in
$\RCA^*_0 +  \bd \Sigma^0_n + \psi$.
Let $(M,\X)$ be a countable model of that theory.
Let $X,Z,A \in \X$ be such that $\neg\ind\Sigma_{n}(A)$ holds and $X\ll^{n}_{A} Z$.
Since $\psi$ is true in $(M,\X)$, there is a set $Y \in \X$ such that $(M,\X) \models \alpha(X,Y)$.
By Lemma~\ref{lem:d0n-wkl-cons}, we can $\bbomega$-extend 
\leszek{$(M, X\oplus Y \oplus Z \oplus A)$} 
to a structure $(M, \Y)$ satisfying $\Delta^0_n\textrm{-}\WKL$.
Of course, $(X\oplus Y \oplus A)^{(n-1)}$ is a $\Delta^0_n$-set in $(M,\Y)$.

The statement $X\ll^{n}_{A} Z$ is arithmetical in $X,Z,A$, so since it was true in $(M,\X)$,
it also holds in all the other second-order universes considered.
So, Lemma \ref{lem:ll-facts}\ref{pt:ll:cod} implies that
there is a $\Delta_n({Z\oplus A})$-set $\extset W$ coding a model $(M,\extset \W)$
of $\WKL^*_0$ with $(X\oplus A)^{\leszek{(}n-1\leszek{)}} \in \extset \W$.
By $\Delta^0_n\textrm{-}\WKL$ (and Proposition \ref{prop:n-wkl-bd}), we know that $(M,\Delta^0_n\textrm{-Def}(M,\Y))$
is a model of $\WKL^*_0$, and by $\neg\ind\Sigma_{n}(A)$, we have $(M,\Delta^0_n\textrm{-Def}(M,\Y))\models \neg \ind \Sigma^0_1$ and $(M,\extset \W)\models \neg \ind \Sigma^0_1$.

Thus, by Corollary \ref{cor:wkl-elem-equiv},
the statement ``there exists $Y$ such that $\alpha(X,Y)$ and $(X\oplus Y \oplus A)^{(n-1)}$ exists as a set'',
which is true in $(M,\Delta^0_n\textrm{-Def}(M,\Y))$,
must also be true in $(M,\extset \W)$. Let $\extset Y\in\extset \W$ be a witness to the $\exists Y$ quantifier
in that statement.
Then in $(M, X \oplus Z \oplus A)$, and hence also in $(M,\X)$,
it is the case that $\extset Y$ is a $\Delta_{n}({Z\oplus A})$-set and
$\alpha(X,\extset Y)$ holds. Moreover, $X\oplus \extset Y\ll_{A}^{n} Z$,
because each $\Delta_n(X\oplus \extset Y \oplus A)$-set belongs to $\extset \W$
which is an $\bbomega$-model of $\WKL^*_0$ coded by a $\Delta_n(Z \oplus A)$-set (cf.~the proof of the
right-to-left direction of Lemma \ref{lem:ll-facts}\ref{pt:ll:cod}, whose statement might not apply directly
because $\extset Y$ might not be a set in the sense of $(M,\X)$).

Since $X, Z, A$ were arbitrary such that $\neg\ind\Sigma_{n}(A)$ holds and $X\ll^{n}_{A} Z$,
this completes the argument that $\RCA^*_0 + \bd\Sigma^0_n + \psi$
proves $\gamma_\psi$. By $\Pi^1_1$-conservativity, also $\RCA^*_0 + \bd\Sigma^0_n + \neg \ind \Sigma^0_n$
proves $\gamma_\psi$. Of course, $\RCA^*_0 + \ind \Sigma^0_n$ proves $\gamma_\psi$
as well by the definition of $\gamma_{\psi}$.

In the right-to-left direction, assume that $\RCA^*_0 + \bd \Sigma^0_n$ proves $\gamma_\psi$.
By a standard $\omega$-chain argument,
to prove $\Pi^1_1$-conservativity of $\psi$ over $\RCA^*_0 + \bd\Sigma^0_n + \neg \ind \Sigma^0_n$
it is enough to show that for any countable $(M,A) \models \bd\Sigma^0_n + \neg \ind \Sigma_n(A)$
and any $\Delta_1(A)$-set $X$, there exists $\extset Y \subseteq M$
such that $(M,\extset Y\oplus A) \models \bd \Sigma^0_n$ and $\alpha(X,\extset Y)$ holds.
By Lemma \ref{lem:d0n-wkl-cons},
we can extend $(M,A)$ to a model $(M, \X) \models \RCA^*_0 + \Delta^0_n\textrm{-}\WKL$.
By Lemma \ref{lem:ll-facts}\ref{pt:ll:WKL}, we can take some $Z\in\X$ such that 
\leszek{$X\ll^{n}_{A}Z$.} 
It follows from our assumption that $(M,\X) \models \gamma_\psi$, so there is some $\extset Y \subseteq M$
such that $\alpha(X,\extset Y)$ holds and 
\leszek{$X \oplus \extset Y \ll^{n}_{A} Z$.} 
Since $Z\oplus A\in \X$, we have $(M,(Z\oplus A)^{(n-1)}) \models \bd\Sigma^0_1$ and thus $(M, \extset Y\oplus A) \models \bd \Sigma^0_n$, 
which is what we wanted to prove.
\end{proof}

We record that the proof of the left-to-right direction of the theorem actually shows the following.

\begin{corollary}\label{cor:proves-gamma}
Let $n \ge 1$ and let $\psi$ be a $\Pi^1_2$ sentence. Then $\RCA^*_0 + \bd\Sigma^0_n + \psi$
proves $\gamma^n_\psi$, where $\gamma^n_\psi$ is the $\Pi^1_1$ sentence from
Theorem~\ref{lem:criterion-pi11-cons-over-bd-not-ind}.
\end{corollary}

A discussion of what Theorem~\ref{lem:criterion-pi11-cons-over-bd-not-ind} and Corollary \ref{cor:proves-gamma}
say about proving conservativity over collection principles, mostly in the context of Ramsey's theorem
for pairs and two colours, can be found in Section \ref{sec:rt22}.

The following Theorem can be viewed as a generalization of Theorem \ref{thm:pi11-cons-over-not-is1}. Whereas
Theorem \ref{thm:pi11-cons-over-not-is1} says that a $\Pi^1_2$ sentence $\psi$ is $\Pi^1_1$-conservative
over $\RCA^*_0 + \neg \ind \Sigma^0_1$ exactly if it is provable from $\WKL$, the result below
replaces $\WKL$ with $\Delta^0_n\textrm{-}\WKL$ and replaces provability of $\psi$ with the provability of a more complicated $\Pi^1_2$ sentence guaranteeing that the second-order universe can be extended by solutions to instances of $\psi$.
The sentence says that well-behaved but possibly non-set solutions to $\psi$ exist,
and that they can be found uniformly for instances of bounded complexity.

\begin{theorem}\label{thm:pi11-cons-over-coh}
Let $n \ge 1$ and let $\psi$ be a $\Pi^1_2$ sentence of the form $\forall X \, \exists Y \, \alpha(X,Y)$ where $\alpha$ is arithmetical.
Then $\psi$ is $\Pi^1_1$-conservative over $\RCA^*_0 + \bd\Sigma^0_n + \neg \ind \Sigma^0_n$ if and only if
$\RCA^*_0 + \Delta^0_n\textrm{-}\WKL + \neg \ind \Sigma^0_n$ proves the statement:
\begin{equation}\label{eqn:low-psi}
\forall X_{0}\,\exists Y_{0}\,\, \forall X \! \le_{\mathrm{T}} \! X_{0}\,\exists \, \Delta_n(Y_{0})\textrm{-}\mathrm{set}\, \extset Y\, \bigl(\extset Y\oplus X_{0}\ll^{n}_{\emptyset}Y_{0} \land \alpha (X,\extset Y)\bigr).
\end{equation}
\end{theorem}

\begin{proof}
Let $\psi$ be $\Pi^1_2$ of the form $\forall X \, \exists Y \, \alpha(X,Y)$.

For the left-to-right direction, assume that $\psi$ is
$\Pi^1_1$-conservative over $\RCA^*_0 + \bd\Sigma^0_n + \neg \ind \Sigma^0_n$,
so $\RCA^*_0 + \bd \Sigma^0_n$ proves the sentence $\gamma_\psi$
from Theorem \ref{lem:criterion-pi11-cons-over-bd-not-ind}.
We argue within $\RCA^*_0 + \Delta^0_n\textrm{-}\WKL + \neg \ind \Sigma^0_n$
and fix a set $A$ satisfying $\neg\ind\Sigma_{n}({A})$.
Consider any set $X_{0}$.
By Lemma \ref{lem:ll-facts}\ref{pt:ll:WKL}, there is $Y_{0}$ such that {$X_{0} \ll^{n}_{X_0 \oplus A}Y_{0}$}.
Naturally, this implies $X \ll^{n}_{X_0 \oplus A}Y_{0}$ for every $X \! \le_{\mathrm{T}} \! X_{0}$ as well.
We have $\neg \ind \Sigma_n(X_0 \oplus A)$,
so we know by $\gamma_{\psi}$ that the $\ll^{n}_{X_0 \oplus A}$-basis theorem for $\psi$ holds.
Therefore, for every $X \! \le_{\mathrm{T}} \! X_{0}$
there exists a $\Delta_n(X_0 \oplus Y_0 \oplus A)$-set $\extset Y$ such that
$\alpha(X, \extset Y)$ and $\extset Y \oplus X \ll^{n}_{X_0 \oplus A} Y_0$;
the latter implies $\extset Y \oplus X_0 \ll^{n}_{\emptyset} X_0 \oplus Y_0 \oplus A$.
This proves (\ref{eqn:low-psi}) with the outermost $\exists$ quantifier witnessed by $X_0 \oplus Y_0 \oplus A$.

The proof of the right-to-left direction is just like the one in Theorem~\ref{lem:criterion-pi11-cons-over-bd-not-ind}.
%
\end{proof}

The first of the two corollaries below follows directly from either Theorem~\ref{lem:criterion-pi11-cons-over-bd-not-ind} or Theorem \ref{thm:pi11-cons-over-coh}. The second follows from their proofs in the right-to-left direction.

\begin{corollary}\label{cor:rec-axiom}
For each $n \ge 1$, the set of $\Pi^1_2$ sentences which are $\Pi^1_1$-conservative over
$\RCA^*_0 + \bd \Sigma^0_n + \neg \ind \Sigma^0_n$ is \leszek{c.e.~}
Thus, it is computably axiomatizable.
\end{corollary}

Except for $n=1$, we do not know whether the set in question is finitely axiomatizable.

\begin{corollary}\label{cor:omega-ext}
Let $n \ge 1$ and let $\psi$ be a $\Pi^1_2$ sentence which is $\Pi^1_1$-conservative over
$\RCA^*_0 + \bd \Sigma^0_n + \neg \ind \Sigma^0_n$.
Any countable topped model of $\RCA^*_0 + \bd\Sigma^0_n + \neg \ind \Sigma^0_n$
can be $\bbomega$-extended to a model of $\RCA^*_0 + \bd \Sigma^0_n + \neg \ind \Sigma^0_n + \psi$.
\end{corollary}

\begin{remark}
\leszek{An analogue of Corollary \ref{cor:omega-ext} for $\RCA_0 + \ind \Sigma^0_n + \neg \bd \Sigma^0_{n+1}$ fails.
One way of showing this is as follows. 
Consider the theory in the language of first-order arithmetic and an additional predicate $\mathrm{Tr}$ with axioms consisting of $\ind \Sigma_n$ and the statement that $\mathrm{Tr}$ is a truth class, i.e.~satisfies Tarski's inductive conditions for a definition of truth for the arithmetical language. It is known that this theory is conservative over $\ind \Sigma_n$ (this was proved for $\PA$ instead of $\ind \Sigma_n$ in \cite{kotlarski-krajewski-lachlan} and generalized to a wider class of theories in \cite{leigh:conservativity-truth}; see also \cite{ev:new-satisfaction}). 
Therefore, every countable recursively saturated model of $\ind \Sigma_n$ admits a truth class, and hence, by the results of \cite{towsner:maximum-conservative} (see Theorem \ref{thm:towsner-expansion} in the present paper), the $\Sigma^1_1$ sentence ``there is a $\Delta^0_{n+1}$-set which is a truth class for the language of first-order arithmetic'' is $\Pi^1_1$-conservative over $\RCA_0 + \ind \Sigma^0_n + \neg \bd \Sigma^0_{n+1}$. On the other hand, it is also known that a nonstandard model of $\ind \Sigma_n$ admitting a truth class has to be recursively saturated (this was originally proved for $\PA$ instead of $\ind \Sigma_n$ in \cite{lachlan:satisfaction-rec-sat}; the new proof in \cite{kw:disjunctions-stopping}, as noted in Remark 1 of that paper, works already over $\ind \Delta_0 + \exp$ and thus over $\ind \Sigma_n$ for each $n \ge 1$). 
As a consequence, a nonstandard model of $\RCA_0 + \ind \Sigma^0_n$ will not $\bbomega$-extend to a model of the above $\Sigma^1_1$ sentence unless it is recursively saturated.}
\end{remark}


\section{Arithmetical consequences of $\RT^2_2$}\label{sec:rt22}

In this section, we study the implications of our results for
the general question of what methods can be used to prove $\Pi^1_1$-conservativity of $\Pi^1_2$ sentences over collection,
and for the concrete problem whether Ramsey's Theorem for pairs and two colours ($\RT^{2}_{2}$)
is $\Pi^1_1$-conservative over $\RCA_0 + \bd \Sigma^0_2$.
This is a major open problem in reverse mathematics, originally posed in \cite{cholak-jockusch-slaman}.

Our main technical observation in this area is an upper bound on the quantifier complexity of the first-order
part of the $\Pi^1_1$ sentences one has to consider to settle the problem about $\RT^2_2$ (Corollary \ref{cor:pi05-cons}).
We give two proofs of the bound, both of which work for a wider class of statements than just $\RT^2_2$.
The first proof is based on Theorem \ref{lem:criterion-pi11-cons-over-bd-not-ind} and Corollary \ref{cor:proves-gamma}.
The second actually avoids the use of Theorem \ref{thm:wkl-iso} and its corollaries altogether,
but relies heavily on the particular syntactic form of $\RT^2_2$.

\subsection{Implications of the isomorphism theorem}\label{subsec:rt22-iso}


In \cite{cholak-jockusch-slaman}, there are two separate approaches used to prove
that every computable colouring $f \colon [\omega]^2 \to 2$
has an infinite homogeneous set $H$ that is low$_2$, which means that $H'' \equiv_{\mathrm{T}} 0''$.
The ``double jump control'' approach is to build $H$ in such a way as to
control $H''$ directly using $0''$. The ``single jump control'' approach
is to show that if $W$ is any set that has PA-degree relative to $0'$ --
that is, to recall, if each $0'$-computable infinite 0--1 tree has a $W$-computable path -- then there is $H$ homogeneous for $f$
such that $W$ computes $H'$, and in fact still has PA-degree relative to $H'$.
This is essentially the unrelativized version of the $\ll^{2}_{\emptyset}$-basis theorem for $\RT^{2}_{2}$,
i.e.~the special case of the $\ll^{2}_{\emptyset}$-basis theorem for $X = \emptyset$,
with $Z$ such that $Z' \equiv_\mathrm{T} W$.
Since by the low basis theorem relativized to $0'$ there is a set $W$
of PA-degree relative to $0'$ such that $W' \le_{\mathrm{T}} 0''$,
single jump control again gives $H$ that is low$_2$.

A relativized version of the double jump control approach was the one
used to prove $\Pi^1_1$-conservativity of $\RT^2_2$ over
$\RCA_0 + \ind \Sigma^0_2$. When $\ind \Sigma^0_2$ fails,
single jump control (or the $\ll^{2}_{\emptyset}$-basis theorem) seems more likely to be applicable,
and indeed it was applied in \cite{csy:conservation-weaker} to prove that
weakenings of $\RT^2_2$ known as $\CAC$ and $\ADS$
are $\Pi^1_1$-conservative over $\RCA_0 + \bd \Sigma^0_2$.
\emph{A priori}, however, it is conceivable that a proof of $\Pi^1_1$-conservativity
of $\RT^2_2$ over $\bd \Sigma^0_2$ could use neither of the two approaches.


What Theorem \ref{lem:criterion-pi11-cons-over-bd-not-ind} shows
is that if a $\Pi^1_2$ sentence $\psi \defeq \forall X \exists Y \alpha$
is $\Pi^1_1$-conservative over $\RCA_0 + \bd \Sigma^0_n + \neg \ind \Sigma^0_n$,
then in principle this \emph{has to} be provable by means of the $\ll^{n}_{A}$-basis theorem for $\psi$, where
$A$ is any witness to $\neg \ind \Sigma^0_n$.
In the specific case of $\RT^2_2$, the relevant $n$ equals $2$, and conservativity over $\ind\Sigma^0_2$
is already known from \cite{cholak-jockusch-slaman}.
Thus, if $\RT^2_2$ is in fact $\Pi^1_1$-conservative over $\bd\Sigma^0_2$,
then in the currently unknown cases it will always be possible to apply
the single jump control argument (relative to a \leszek{set $A$ witnessing} $\neg \ind \Sigma^0_2$).

The results of Section \ref{sec:generalization} also imply that there is always a bound on the complexity of $\Pi^1_1$ sentences we need to study in order to understand if a given $\Pi^1_2$ sentence is $\Pi^1_1$-conservative over $\bd \Sigma^0_n + \neg \ind \Sigma^0_n$. For a $\Pi^1_2$ sentence $\psi$, it follows from Theorem \ref{lem:criterion-pi11-cons-over-bd-not-ind}
and Corollary \ref{cor:proves-gamma} that $\psi$ is $\Pi^1_1$-conservative over $\bd \Sigma^0_n + \neg \ind \Sigma^0_n$ if and only if $\RCA^*_0 + \bd \Sigma^0_n$ proves the sentence $\gamma^n_\psi$.
If $\psi$ is $\forall \exists \Pi^0_k$, then $\gamma^n_\psi$ can be written as a $\forall \Pi^0_{\ell}$ statement for $\ell = \max(n+3,k+2)$; and, by Corollary~\ref{cor:proves-gamma}, it is provable in $\RCA^*_0 + \bd \Sigma^0_n$. So, we get:

\begin{corollary}
Let $n \ge 1$, and let $\psi$ be a $\forall \exists \Pi^0_k$ sentence, where $k \ge 2$.
Then $\psi$ is $\Pi^1_1$-conservative over $\RCA^*_0 + \bd \Sigma^0_n + \neg \ind \Sigma^0_n$
if and only if it is $\forall \Pi^0_{\ell}$-conservative over $\RCA^*_0 + \bd \Sigma^0_n + \neg \ind \Sigma^0_n$, where $\ell = \max(n+3,k+2)$.
\end{corollary}

$\RT^2_2$ is a $\forall \exists \Pi^0_2$ sentence, and we have the additional information
that it is $\Pi^1_1$-conservative over $\ind \Sigma^0_2$. This gives:

\begin{corollary}\label{cor:pi05-cons}
$\RCA_0 + \RT^2_2$ is $\Pi^1_1$-conservative over $\bd \Sigma^0_2$ if and only
if it is $\forall \Pi^0_5$-conservative over $\bd \Sigma^0_2$.
\end{corollary}

The discussion up to this point has only used the following properties of $\RT^2_2$:
it is a $\Pi^1_2$, and more precisely $\forall \exists \Pi^0_2$, sentence
that is $\Pi^1_1$-conservative over $\ind \Sigma^0_2$ and implies $\bd \Sigma^0_2$ in the presence of $\RCA_0$.
More specific features of $\RT^2_2$ are needed for the remark below.


\begin{remark}
In the special case of $\RT^{2}_{2}$,
the role of $\gamma_{\RT^2_2}^{2}$ in the preceding
can be played by the $\ll^{2}_{\emptyset}$-basis theorem for $\RT^{2}_{2}$.
Indeed, we have the following:
\begin{enumerate}[(1)]
 \item $\RCA_{0}+\RT^{2}_{2}$ proves the $\ll^{2}_{\emptyset}$-basis theorem for $\RT^{2}_{2}$.
 \item $\RCA_0 + \RT^2_2$ is $\Pi^1_1$-conservative over $\bd \Sigma^0_2$ if and only
if $\bd \Sigma^0_2$ proves the $\ll^{2}_{\emptyset}$-basis theorem for $\RT^{2}_{2}$.
\end{enumerate}

To show (1), let $(M,\X)$ be a countable model of $\RCA_{0}+\RT^{2}_{2}$,
and assume that $X,Z\in \X$ are such that $X\ll_{\emptyset}^{2} Z$
and $X$ is a $2$-coloring of $[M]^2$.
(Note that $(M,\X)$ is already a model of $\Delta^0_2\textrm{-}\WKL$.)
By Lemma \ref{lem:ll-facts}\ref{pt:ll:cod},
there is a $\Delta_2(Z)$-set $\extset W$ coding a model $(M,\extset \W)$
of $\WKL^*_0$ with $X' \in \extset \W$.
Applying Lemma \ref{lem:wkl-coded}, take $\extset V\in\W$ such that $X'\ll_{\emptyset}^{1} \extset V$.
If $\ind\Sigma_{1}(\extset V)$ holds, then $\extset V$-primitive recursion is available in $(M,\X)$.
Thus one can directly formalize the single jump control argument
in the form of the proofs of Theorems~6.44 and 6.57 and Corollary~6.58 of \cite{hirschfeldt:slicing},
which are carried out using solely ``$d$-computable/primitive-recursive arguments''
where $d$ is a fixed degree such that $\emptyset'\ll d$.
This gives a $\Delta_{1}(\extset V)$-set $\extset Y$ such that $\extset Y$ is an infinite homogeneous set for $X$ and $(X\oplus \extset Y)'\le_{T} \extset V$.
Since $\extset V \ll_{\emptyset}^{1}Z'$, we get $X\oplus \extset Y\ll_{\emptyset}^{2}Z$.
If $\ind\Sigma_{1}(\extset V)$ fails, then both $(M,\extset \W)$ and $(M,\Delta^0_2\textrm{-Def}(M,\X))$
are models of $\WKL^*_0+\neg\ind\Sigma^{0}_{1}$ containing $\extset V$,
so one may find $\extset Y\in\extset \W$ such that $\extset Y$ is an infinite homogeneous set for $X$ and $(X\oplus \extset Y)'\in \extset\W$ as in the proof of Theorem~\ref{lem:criterion-pi11-cons-over-bd-not-ind}.
This again implies $X\oplus \extset Y\ll_{\emptyset}^{2}Z$.
Hence the $\ll^{2}_{\emptyset}$-basis theorem for $\RT^{2}_{2}$ holds in $(M,\X)$.

The left-to-right direction of (2) is a direct consequence of (1), and the right-to-left direction of (2) can be shown as in the proof of the right-to-left direction of Theorem~\ref{lem:criterion-pi11-cons-over-bd-not-ind}.
\end{remark}

\subsection{Restricted $\Sigma^1_1$ formulas}

\begin{definition}
A \emph{restricted $\Sigma^1_1$}, or $\rs$, formula has the form $\exists Y\xi$, where $\xi$ is $\Sigma^0_3$.
A \emph{restricted $\Pi^1_2$}, or $\rp$, sentence has the form $\forall X\, (\eta(x) \to \exists Y \xi(X,Y))$,
where $\exists Y\xi(X,Y)$ is $\rs$.
\end{definition}

As mentioned above, $\RT^2_2$ is a $\forall \exists \Pi^0_2$ and thus an $\rp$ sentence.

In this subsection, we study the behaviour of $\rs$ formulas in $\WKL^*_0 + \neg \ind \Sigma^0_1$.
We show that each $\rs$ formula is equivalent in $\WKL^*_0 + \neg \ind \Sigma^0_1$
to an arithmetical formula with a relatively clear combinatorial meaning.
Lifted to $n =2$, this provides us with a $\forall \Pi^0_5$ consequence of $\RT^2_2$
that leads to an alternative proof of Corollary \ref{cor:pi05-cons},
but differs from the sentence $\gamma^2_{\RT^2_2}$ used in the proof from Section \ref{subsec:rt22-iso}
by having a meaningful restriction to computable instances.

The investigation of $\rs$ formulas also gives us an opportunity
to discuss some additional interesting properties of $\RCA^*_0 + \neg \ind \Sigma^0_1$.
We begin with a variation on a theme suggested by
Theorem \ref{thm:wkl-iso}: the possibilities for changing a model
of $\RCA^*_0 + \neg \ind \Sigma^0_1$ by adding new sets to it are rather limited.
Here, we show that it is not possible to add an unbounded set that is ``sparser''
than all those present in the ground model or to
add a new bounded $\Sigma^0_1$-definable set. The statement about bounded sets is not
used elsewhere in the paper but it is potentially of independent interest.

\begin{lemma}\label{lem:sparse-set}
Let $(M,\X) \models \RCA^*_0$ and let $A \in \X$ be such that $(M,A) \models \neg \ind \Sigma_1(A)$.
Then:
\begin{enumerate}[(a)]
\item\label{pt:sparse:sparse} For every unbounded set $S \in \X$ there exists an unbounded
$\Delta_1(A)$-definable set $B$ such that
for every $b_1,b_2 \in B$ with $b_1 < b_2$ there exist
$s_1, s_2 \in S$ with $b_1 \le s_1 < s_2 \le b_2$.
\item\label{pt:sparse:expansion} Every bounded $\Sigma^0_1(M,\X)$-definable set
 is $\Sigma_1(A)$-definable.
\end{enumerate}
\end{lemma}

The result has a rather obvious generalization to higher $n$, in which $(M,\X)$ satisfies
$\bd\Sigma^0_n + \neg \ind \Sigma_n(A)$ and the statements \ref{pt:sparse:sparse} and \ref{pt:sparse:expansion} concern
$\Delta^0_n$- and $\Sigma^0_n$-definable sets, respectively.

\begin{proof}
Let $C$ be an unbounded $\Delta_1(A)$-definable set such that $C = \{c_i: i \in I\}$ for some proper $\Sigma_1(A)$-definable cut $I \subsetneq_\ee M$.
Such a set $C$ exists because $\neg \ind \Sigma_1(A)$ holds.

We first prove \ref{pt:sparse:sparse}. Given an unbounded $S \in \X$, we find the set $B$ as a suitable subset of $C$.
Define a sequence $\langle i_k: k \in K\rangle$ cofinal in $I$ by
\begin{align*}
i_{0}&=0, \\
i_{k+1}&=\min\{i>i_{k}: |S\cap [c_{i_k}, c_{i}]|\ge 2\}.
\end{align*}
Here $K$ consists of exactly those $k \in I$ for which $i_k$ exists.
$K$ is clearly $\Sigma^0_1$-definable, and it is a cut because both $S$ and $C$ are unbounded sets.
Let $B$ be $\{c_{i_k}: k \in K\}$.

Clearly, $B$ is unbounded and there are at least two elements of $S$ between
any two elements of $B$. It remains to show that $B$ is $\Delta_1(A)$-definable.
Notice that both $\{ i_k: k \in K \}$ and $I \setminus \{ i_k: k \in K \}$ are $\Sigma^0_1$-definable.
So, by Theorem \ref{lem:chong-mourad}, there is some $d \in M$
such that $\Ack(d) \cap I = \{ i_k: k \in K \}$.
Then $B = \{c_i \in C: i \in \Ack(d)\}$, which shows that $B$ is $\Delta_1(A)$-definable because $C$ is.

For \ref{pt:sparse:expansion}, we adopt some ideas from the proof of Theorem~2.2
 in Chong--Yang~\cite{art:cy}.
Let $R$ be a $\Sigma^0_1$-definable set in $(M,\X)$
 that is bounded above by $e\in M$.
Suppose $R=\{x\in M : (M,\X)\models\exists z\, \theta(x,z)\}$,
 where $\theta$ is a $\Sigma^0_0$~formula,
  possibly with parameters from~$(M,\X)$.
Use Theorem~\ref{lem:chong-mourad} to obtain $d\in M$
 such that
 \begin{equation*}
  \Ack(d)\cap(I\times[0,e]) = \{ \tuple{i,x} :
   (M,\X)\models\eb{z}{c_i} \theta(x,z)
  \}.
 \end{equation*}
Then
 \begin{math}
  R=\{x\in M:x\le e \land \exists i\!\in\! I\,
   \tuple{i,x}\in\Ack(d)
  \}
 \end{math}.
So $R$ is $\Sigma_1(A)$-definable.
\end{proof}

In general, an $\rs$ formula has the form $\exists Y \, \exists w\, \forall u\, \exists v\, \delta$ where
$\delta$ is bounded. However, we can assume without loss of generality that the arithmetical part of the formula
is actually $\forall u\, \exists v\, \delta$, because the initial existential quantifier $\exists w$ can be merged
with the existential set quantifier $\exists Y$. Using standard tricks, we can also assume
that all quantifiers in $\delta$ are bounded by $v$.

\begin{definition}\label{def:alpha-sigma}
Let $\varphi(X)$ be an $\rs$ formula of the form $\exists Y \, \forall u\, \exists v\, \delta(X,Y, u,v)$,
where all quantifiers in $\delta$ are bounded by $v$. We define $\alpha_\varphi(X,Z)$
to be the following arithmetical statement:
\begin{quote}
There is an unbounded $\Delta_1(Z)$-definable set $W = \{w_i : i \in I\}$ \\
such that for every $i \in I$ there is a finite set $y \subseteq [0,w_i]$ \\
satisfying $\ab{u}{w_j}\eb{v}{w_{j+1}}\delta(X,y,u,v)$ for each $j < i$.
\end{quote}
\end{definition}

Roughly speaking, $\alpha_\varphi(X,Z)$ says that some $Z$-computable set
provides a lower bound for the ``rate of convergence'' of sequence of finite approximations
to a witness for the $\exists Y$ quantifier in $\varphi(X)$.

\begin{theorem}\label{thm:rs-arithmetical}
Let $\varphi(X)$ be $\rs$, and let $\alpha_\varphi(X,Z)$ be the arithmetical formula
from Definition \ref{def:alpha-sigma}. Then:
\begin{enumerate}[(a)]
\item\label{pt:rs>} $\RCA^*_0 \vdash \forall X\, \forall Z \, (\neg \ind \Sigma_1(Z) \to (\varphi(X) \to \alpha_\varphi(X,Z)))$,
\item\label{pt:rs<} $\WKL^*_0 \vdash \forall X\, \forall Z \, (\alpha_\varphi(X,Z) \to \varphi(X))$.
\end{enumerate}
\end{theorem}

\begin{proof}
Let $\varphi(X)$ be $\exists Y \, \forall u\, \exists v\, \delta(X,Y, u,v)$
with quantifiers in $\delta$ bounded by $v$.

We first prove \ref{pt:rs<}. Let $X, Z$ be such that $\alpha_\varphi(X,Z)$ holds,
and let the set $W = \{w_i : i \in I\}$ witness the existential quantifier in $\alpha_\varphi(X,Z)$.
Let $T$ be the tree consisting of finite 0--1 strings $\tau$ such that for every $i$ satisfying
$w_{i+1} < |\tau|$, we have $\ab{u}{w_i}\eb{v}{w_{i+1}}\delta(X,\{x:\tau(x)=1\},u,v)$.
By the choice of $W$, the tree $T$ is infinite, so by $\WKL$ there is an infinite
path $Y$ in $T$. Clearly, we have $\ab{u}{w_i}\eb{v}{w_{i+1}}\delta(X,Y,u,v)$
for each $i \in I$, which implies $\varphi(X)$.

Turning to \ref{pt:rs>}, assume that we have $\varphi(X)$, and let $Y$ witness the existential set quantifier in $\varphi(X)$.
By $\bd\Sigma^0_1$, for every number $a$ there is some $b$ such that
$\ab{u}{a}\eb{v}{{b}}\delta(X,Y,u,v)$. This implies the existence
of an unbounded set $\widehat W = \{\widehat w_j : j \in J\}$ such that
for every $j \in J$ we have $\ab{u}{\widehat w_j}\eb{v}{\widehat w_{j+1}}\delta(X,Y,u,v)$.

Now let $Z$ be such that $\neg \ind\Sigma_1(Z)$ holds.
By Lemma \ref{lem:sparse-set}\ref{pt:sparse:sparse} with $S\defeq \widehat W$,
there exists an unbounded $\Delta_1(Z)$-definable set $W = \{w_i: i \in I\}$
such that for each $i \in I$ there is some $j \in J$ with $w_i \le \widehat w_j < \widehat w_{j+1} \le w_{i+1}$.
As a consequence, we have $\ab{u}{\widehat w_i}\eb{v}{\widehat w_{i+1}}\delta(X,Y,u,v)$ for each $i \in I$.
But this means in particular that $W$ has the property required in $\alpha_\varphi(X,Z)$.
\end{proof}

\begin{corollary}Let $\psi$ be an $\rp$ sentence of the form
$\forall X\, (\eta(x) \to \varphi(X))$, where $\varphi$ is $\rs$.
Then $\psi$ is $\Pi^1_1$-conservative over $\RCA^*_0 + \neg \ind \Sigma^0_1$
if and only if $\RCA^*_0$ proves $\forall X\, \forall Z \, (\neg \ind \Sigma_1(Z) \land \eta(X) \to \alpha_\varphi(X,Z))$.
\end{corollary}


We now consider what can be said about $\RT^2_2$ using techniques based on Lemma \ref{lem:sparse-set}.
As in the proof of Corollary \ref{cor:pi05-cons} in Section \ref{subsec:rt22-iso},
the only specific features of $\RT^2_2$ needed below are that it is an $\rp$ sentence,
implies $\bd\Sigma^0_2$, and is $\Pi^1_1$-conservative over $\RCA_0 + \ind \Sigma^0_2$.

Let $\zeta(f,Z)$ express the following:

\begin{center}
\begin{minipage}{0.85\textwidth}
If $f \colon [\N]^2 \to 2$, then there is an unbounded $\Delta_2(Z)$-set \mbox{$\extset W = \{w_i : i \in I\}$}
such that for every $i \in I$ there is a pair of finite strings $\tuple{\sigma,\tau}$
satisfying:
\begin{itemize}
\item $|\sigma| = |\tau| = w_i$,
\item $\{x: \sigma(x) = 1\}$ is homogeneous for $f$,
\item for each $j < i$, there are at least $w_j$ elements $x \le w_{j+1}$ such that $\sigma(x) = 1$,
\item for each $j < i$ and each $e \le w_j$, if $\tau(e) = 1$, then there is a computation $s \le w_{j+1}$
witnessing $e^{Z \oplus \sigma}{\downarrow}$, and if $\tau(e) = 0$, then $e^{Z \oplus \sigma}{\uparrow}$.
\end{itemize}
\end{minipage}
\end{center}

Loosely speaking, $\zeta$ says that, assuming $f$ is a $2$-colouring of pairs,
there is a $\Delta_2(Z)$-definable infinite tree of finite approximations
to an infinite homogeneous set for $f$ and to the jump of the join of that set with $Z$.
Note that $\zeta$ can be written as a $\Sigma^0_4$ formula, so the statement
$\forall f\, \forall Z\,(\neg \ind \Sigma_2(Z) \to \zeta(f,Z))$ is $\forall \Pi^0_5$.

\begin{theorem}\label{thm:rt22-proves-apx}
$\RCA_0 + \RT^2_2$ proves $\forall f\, \forall Z\,(\neg \ind \Sigma_2(Z) \to \zeta(f,Z))$.
$\RCA_0 + \bd\Sigma^0_2$ proves that statement if and only if
$\RT^2_2$ is $\Pi^1_1$-conservative over $\bd\Sigma^0_2$.
\end{theorem}

\begin{proof}
We first show that $\RCA_0 + \RT^2_2$ proves $\forall f\, \forall Z\,(\neg \ind \Sigma_2(Z) \to \zeta(f,Z))$.
Let $(M,\X) \models \RCA_0 + \RT^2_2$. Let $f  \in \X$ be a $2$-colouring of pairs from $M$,
and let $Z \in \X$ be such that $\ind \Sigma_2(Z)$ fails.
By $\RT^2_2$, there is $H \in \X$ which is an infinite homogeneous set for $f$.
Since $(M, \X)$ satisfies $\bd \Sigma^0_2$, there is an unbounded $\Delta^0_2$-set
$\extset S = \{s_j : j \in J\}$ in $(M,\X)$ such that for each
$j$, there are at least $s_j$ elements of $H$ below $s_{j+1}$,
and each machine $e \le s_j$ run with oracle $H \oplus Z$
either stops before $s_{j+1}$ or does not stop at all.

By Lemma \ref{lem:sparse-set}\ref{pt:sparse:sparse} applied to $(M,\Delta^0_2\textrm{-Def}(M, \X))$
with $S: =\extset S$ and $A \defeq Z'$,
there exists an unbounded $\Delta_2(Z)$-definable
set $\extset W = \{w_i : i \in I\}$ such that there are at least two
elements of $\extset S$ between any two elements of $\extset W$.
We claim that $\extset W$ witnesses that $\zeta(f,Z)$ holds.
To see this, consider fixed $i \in I$.
Let $\sigma$ be the characteristic function of $H \cap [0,w_i]$,
and let $\tau$ be the characteristic function of $(H \oplus Z)' \cap [0,w_i]$.
Then the pair $\tuple{\sigma,\tau}$ has the properties required by $\zeta(f,Z)$ for this $i$.
This shows that $(M, \X) \models \forall f\, \forall Z\,(\neg \ind \Sigma_2(Z) \to \zeta(f,Z))$.

Thus, we have proved that $\RCA_0 + \RT^2_2$ implies $\forall f\, \forall Z\,(\neg \ind \Sigma_2(Z) \to \zeta(f,Z))$.
Now assume that $\RCA_0 + \bd \Sigma^0_2$ proves that statement as well.
As in the proof of Lemma \ref{lem:criterion-pi11-cons-over-bd-not-ind},
to prove the $\Pi^1_1$-conservativity of $\RT^2_2$ over $\RCA_0 + \bd \Sigma^0_2 + \neg \ind \Sigma^0_2$,
and thus over $\RCA_0 + \bd \Sigma^0_2$,
it is enough to show that for any countable $(M,A) \models \bd\Sigma^0_2 + \neg \ind \Sigma^0_2$
and any $\Delta_1(A)$-definable $2$-colouring of pairs $f$,
there is $\extset H \subseteq M$ unbounded homogeneous for $f$
such that $(M, A \oplus \extset H)$ still satisfies $\bd\Sigma^0_2$.

By \cite{csy:conservation-weaker}, there is a model $(M,\X) \models \RCA_0 + \bd\Sigma^0_2 + \COH$
with $A \in \X$. By our assumption, $\zeta(f,A)$ holds in $(M,\X)$. Since
$\COH$ implies $\Delta^0_2$-$\WKL$ over $\RCA_0 + \bd\Sigma^0_2$,
there is a $\Delta^0_2$-set in $(M,\X)$ which
is an infinite path in the infinite $\Delta_2$-definable 0--1 tree provided by $\zeta(f,A)$.
Thus, there is an unbounded $\Delta^0_2$-set $\extset H$
in $(M,\X)$ such that $\extset H$ is homogeneous for $f$
and $(A \oplus \extset H)'$ is a $\Delta^0_2$-set in $(M,\X)$.
But this means that $(M,\leszek{(}A \oplus \extset H)'\leszek{)} \models \bd\Sigma^0_1$,
so $(M,A \oplus \extset H) \models \bd\Sigma^0_2$.
\end{proof}


By Theorem \ref{thm:rt22-proves-apx},
the statement $\forall f\, \forall Z\,(\neg \ind \Sigma_2(Z) \to \zeta(f,Z))$
can be used to give an alternative proof of Corollary \ref{cor:pi05-cons}.
A possible advantage of that statement over the sentence $\gamma_{\RT^2_2}$
is that the restriction of the latter to computable instances is trivially true in any model of $\bd \Sigma_1 + \exp$,
because in such a model there can never be computable sets $A,Z,X$
such that $X \ll^n_A Z$ (cf.~Lemma \ref{lem:non-arith}).
On the other hand, the $\Pi_5$ sentence obtained by restricting $\zeta(f, Z)$
to computable colourings $f$ and computable $Z$
(that is, essentially, to situations where $\ind \Sigma_2$ fails)
seems more interesting, and it is quite unclear whether
$\bd \Sigma_2$ proves it.


\section{Solution to a problem of Towsner}\label{sec:towsner}

In \cite{towsner:maximum-conservative}, Towsner proved that for every
$n \ge 1$, the set of $\Pi^1_2$ sentences $\psi$ which are $\Pi^1_1$-conservative
over $\RCA_0 + \ind\Sigma^0_n$ is $\Pi_2$-complete.

Towsner also asked whether the same holds with $\ind \Sigma^0_n$ replaced by $\bd \Sigma^0_n$.
The question as stated is only really meaningful for $n \ge 2$, because
$\RCA_0 + \bd \Sigma^0_1$ is simply $\RCA_0$. So, in order to generalize
the question to $n = 1$ we take the liberty of changing the base theory to $\RCA^*_0$.

\begin{question}[essentially Towsner \cite{towsner:maximum-conservative}]\label{q}
For fixed $n$, is the set
\begin{center}
$\{\psi \in \Pi^1_2: \RCA^*_0 + \bd\Sigma^0_n + \psi \textrm{ is } \Pi^1_1\textrm{-conservative over }
\RCA^*_0 + \bd\Sigma^0_n\}$
\end{center}
$\Pi_2$-complete?
\end{question}

By Theorem \ref{thm:pi11-cons-over-not-is1} and Corollary \ref{cor:rec-axiom},
the answer to Question \ref{q} is ``almost negative'', in that
the set of $\Pi^1_2$ sentences $\psi$ which are $\Pi^1_1$-conservative
over $\RCA^*_0 + \bd \Sigma^0_n + \neg \ind\Sigma^0_n$ is \leszek{c.e.~}
for each $n$.
For $n = 1$, we even know that this set is finitely axiomatizable.

Below we show that the original question,
without the explicitly added $\neg \ind\Sigma^0_n$, nevertheless has a positive answer.

\begin{theorem}\label{thm:positive-solution}
For every $n \ge 1$, the set
\begin{center}
$\{\psi \in \Pi^1_2: \RCA^*_0 + \bd\Sigma^0_n + \psi \textrm{ is } \Pi^1_1\textrm{-conservative over }
\RCA^*_0 + \bd\Sigma^0_n\}$
\end{center}
is $\Pi_2$-complete.
\end{theorem}

To prove this, we recall an important result that Towsner uses as a lemma
in his argument for the $\Pi_2$-completeness of $\Pi^1_1$-conservativity
over $\ind \Sigma^0_n$.

\begin{theorem}\cite{towsner:maximum-conservative}\label{thm:towsner-expansion}
If $(M,\X)$ is a countable model of $\RCA_0 + \ind\Sigma^0_n$ and $\extset S \subseteq M$,
then there is a family $\Y\supseteq\X$ of subsets of $M$ such that $(M,\Y) \models \RCA_0 + \ind\Sigma^0_n$
and $\extset S$ is $\Delta^0_{n+1}$-definable in $(M,\Y)$.
\end{theorem}

\begin{definition}
For each $n \ge 1$, the $\Sigma^0_n$ \emph{cardinality scheme},
$\mathrm{C}\Sigma^0_n$, asserts that no $\Sigma^0_n$ formula
defines a total injection from $\N$ to $\N$ with bounded range.
\end{definition}

\begin{theorem}\label{thm:neg-is1-cs2}\cite{kky:ramsey-rca0star}
For each $n,k \ge 1$, the theory $\RCA^*_0 + \bd\Sigma^{0}_{n} + \neg\ind\Sigma^0_{n}$
proves $\mathrm{C}\Sigma^0_k$.
\end{theorem}

A proof of Theorem \ref{thm:neg-is1-cs2} is given in \cite{kky:ramsey-rca0star}. As mentioned in \cite{kky:ramsey-rca0star},
a different proof from the one described there can be obtained by relativizing Kaye's proof of the result  that any model of $\bd\Sigma_1 + {\exp} + {\neg\ind\Sigma_1}$ is elementarily equivalent to an $\aleph_\omega$-like structure \cite[Theorem 2.4]{kaye:constructing-kappa-like}.

Note that Theorem \ref{thm:neg-is1-cs2} means that an analogue of Theorem \ref{thm:towsner-expansion}
\leszek{for $\bd \Sigma^0_n$} fails. For example, if $(M,\X)$ is a countable model of $\RCA^*_0 + \bd \Sigma^0_n + \neg \ind\Sigma^0_n$
then the graph of any bijection between $M$ and the standard cut $\omega$ cannot be arithmetically
definable in any $(M, \Y) \models \bd \Sigma^0_n$ with $\X \subseteq \Y$.

In the proof of Theorem \ref{thm:positive-solution},
we use a combination of Theorem \ref{thm:towsner-expansion} and Theorem \ref{thm:neg-is1-cs2}.

\begin{proof}[Proof of Theorem \ref{thm:positive-solution}]
Fix $n \ge 1$.
The set of those $\Pi^1_2$ sentences $\psi$ that are $\Pi^1_1$-conservative over $\RCA^*_0 + \bd\Sigma^0_n$
is clearly $\Pi_2$, so we only need to prove completeness.
Given a $\Pi_2$ sentence $\varphi \defeq \forall x\, \exists y\, \delta(x,y)$, define the $\Sigma^1_1$ sentence $\psi_\varphi$
as
\[\neg \ind \Sigma^0_n \lor \exists Z\,\exists a\,\left[(\textrm{there exists a } \Sigma_{n+1}(Z) \textrm{ function } \extset f \colon \N \hookrightarrow a) \land \ab{x}{a}\exists y\,\delta(x,y)\right] .\]
We claim that $\psi_\varphi$ is $\Pi^1_1$-conservative over $\RCA^*_0 + \bd\Sigma^0_n$ if and only if
$\varphi$ is true in the standard model of arithmetic $\omega$.

Assume that $\varphi$ is true in $\omega$, and let $(M, \X)$ be a countable nonstandard
model of $\RCA^*_0 + \bd \Sigma^0_n$.
If $(M,\X) \models \neg \ind \Sigma^0_n$, then also $(M,\X) \models \psi_\varphi$. On the other hand,
if $(M,\X) \models \ind \Sigma^0_n$, by Theorem \ref{thm:towsner-expansion} we can $\bbomega$-extend $(M,\X)$ to $(M,\Y)$ in which there is a $\Delta^0_{n+1}$-definable (and thus $\Sigma^0_{n+1}$-definable) bijection $\extset f$ between $M$ and $\omega$.
If $M \models \varphi$, let $a$ be any nonstandard element of $M$. If $M \models \neg \varphi$, then, using
$\ind\Sigma_1$ in $M$, let $a$ be the largest element of $M$ such that $M \models \ab{x}{a}\exists y\,\delta(x,y)$;
in this case, $a$ is necessarily nonstandard. In each case, we see that $\extset f$ is an injection
from $M$ into $a_M$, so $(M,\Y) \models \psi_\varphi$.

We have shown that if $\omega \models \varphi$, then every countable nonstandard model
of $\RCA^*_0 + \bd \Sigma^0_n$ $\bbomega$-extends to a model of $\RCA^*_0 + \bd \Sigma^0_n + \psi_\varphi$.
This is enough to show that $\psi_\varphi$ is $\Pi^1_1$-conservative over $\varphi$.

Now assume that $\varphi$ is false in $\omega$, and let $m \in \omega$ be such that
$\omega \models \neg \exists y\, \delta(m,y)$. We argue that the $\Sigma^1_1$ formula
$\neg \exists y\, \delta(m,y) \land \neg \mathrm{C}\Sigma^0_{n+1}$ is consistent
with $\RCA^*_0 + \bd \Sigma^0_n$ (in fact, with $\RCA^*_0 + \ind \Sigma^0_n$)
but not with $\RCA^*_0 + \bd \Sigma^0_n + \psi_\varphi$.

Let $M$ be a countable nonstandard model of $\mathrm{Th}(\omega)$.
Then $(M,\Delta_{1}\textrm{-Def}(M)) \models \RCA^*_0 + \ind \Sigma^0_n + \neg \exists y\, \delta(m,y)$.
Use Theorem \ref{thm:towsner-expansion} to $\bbomega$-extend $(M,\Delta_{1}\textrm{-Def}(M))$
to $(M,\Y) \models \RCA^*_0 + \ind \Sigma^0_n + \neg \mathrm{C}\Sigma^0_{n+1}$.
Of course, $\neg \exists y\, \delta(m,y)$ still holds in $(M,\Y)$,
because it is a purely arithmetical statement.
This shows the consistency of $\neg \exists y\, \delta(m,y) \land \neg \mathrm{C}\Sigma^0_{n+1}$
with $\RCA^*_0 + \bd \Sigma^0_n$.

On the other hand, assume that a structure $(M,\X)$
satisfies both $\RCA^*_0 + \bd \Sigma^0_n + \psi_\varphi$
and $\neg \exists y\, \delta(m,y) \land \neg \mathrm{C}\Sigma^0_{n+1}$.
By Theorem \ref{thm:neg-is1-cs2}, $(M,\X)$ must satisfy $\ind\Sigma^0_n$,
so it must also satisfy the second disjunct of $\psi_\varphi$.
In particular, for any element $a$ witnessing the existential number quantifier
in that disjunct, we have $M \models \ab{x}{a}\exists y\,\delta(x,y)$.
On the other hand, any such $a$ also has to be nonstandard, which
gives a contradiction with $M \models \neg \exists y\, \delta(m,y)$.
\end{proof}

\section*{Acknowledgements}
Fiori Carones and Kołodziejczyk were partially supported by grant number 2017/27/B/ST1/01951 of the
National Science Centre, Poland.
Part of this research was conducted when Wong was financially supported by the Singapore Ministry of Education Academic Research Fund Tier 2 grant MOE2016-T2-1-019 / R146-000-234-112.
Yokoyama was partially supported by JSPS KAKENHI grant number 19K03601 and 21KK0045.

\bibliographystyle{plain}
\bibliography{leszek2014,lawrence-keita-add}

\end{document}